\documentclass[12pt]{amsart}
\usepackage{SSdefn}
\usepackage{eucal}

\let\wh\widehat
\let\mf\mathfrak

\renewcommand{\fg}{\mathfrak{g}}
\newcommand{\osp}{\mathfrak{osp}}
\newcommand{\Pin}{\mathbf{Pin}}
\newcommand{\Mp}{\mathbf{Mp}}
\newcommand{\FB}{\mathbf{FB}}
\newcommand{\OSp}{\mathbf{OSp}}

\newcommand{\spin}{\mathbf{spin}}
\newcommand{\osc}{\mathbf{osc}}
\newcommand{\dspin}{{(\rd\Delta)}}
\newcommand{\dosc}{{(\rd\nabla)}}
\newcommand{\uspin}{{(\ru\Delta)}}
\newcommand{\uosc}{{(\ru\nabla)}}
\newcommand{\db}{{({\rm db})}}
\newcommand{\ub}{{({\rm ub})}}
\newcommand{\dsb}{{({\rm dsb})}}
\newcommand{\usb}{{({\rm usb})}}
\renewcommand{\Vec}{\mathrm{Vec}}
\DeclareMathOperator{\Cl}{Cl}

\title{Infinite rank spinor and oscillator representations}

\author{Steven V Sam}
\address{Department of Mathematics, University of Wisconsin, Madison, WI}
\email{\href{mailto:svs@math.wisc.edu}{svs@math.wisc.edu}}
\urladdr{\url{http://math.wisc.edu/~svs/}}

\author{Andrew Snowden}
\address{Department of Mathematics, University of Michigan, Ann Arbor, MI}
\email{\href{mailto:asnowden@umich.edu}{asnowden@umich.edu}}
\urladdr{\url{http://www-personal.umich.edu/~asnowden/}}

\date{February 24, 2017}

\thanks{SS was partially supported by a Miller research fellowship and NSF grant DMS-1500069. AS was supported by NSF grants DMS-1303082 and DMS-1453893.}

\begin{document}

\begin{abstract}
We develop a functorial theory of spinor and oscillator representations parallel to the theory of Schur functors for general linear groups. This continues our work on developing orthogonal and symplectic analogues of Schur functors. As such, there are a few main points in common. We define a category of representations of what might be thought of as the infinite rank pin and metaplectic groups, and give three models of this category in terms of: multilinear algebra, diagram categories, and twisted Lie algebras. We also define specialization functors to the finite rank groups and calculate the derived functors.
\end{abstract}

\maketitle

\setcounter{tocdepth}{1}

\tableofcontents

\section{Introduction}

This paper is a continuation of \cite{infrank, lwood}. Let us briefly recall the setup and motivation. The theory of polynomial representations of $\GL_n(\bC)$ is closely tied to the theory of polynomial functors, or equivalently, the theory of polynomial representations of the infinite general linear group $\GL_\infty(\bC)$. The main important property is that polynomial functors can be evaluated on a vector space, and this evaluation is an exact functor. When one attempts to generalize this theory to the orthogonal and symplectic groups, a few complications arise. In \cite{infrank}, we introduced an analogue of the category of polynomial functors for these groups along with an evaluation functor; however, it fails to be right-exact (or left-exact, depending on the conventions used). In \cite{lwood}, we compute the higher derived functors on simple objects. In fact, the calculation of the Euler characteristic of these functors appeared in the literature much earlier in \cite{koiketerada}, though not under this setup.

While this accounts for all finite-dimensional representations of the orthogonal and symplectic groups, there is an issue remaining: the orthogonal group has a double cover, the pin group, which has many more finite-dimensional representations (in some sense, the representations that factor through the orthogonal group only account for ``half'' of them), the so-called spinor representations. 

The purpose of this paper is to develop an analogous functorial theory for spinor representations. In fact, when dealing with spinor representations from a multilinear perspective, as we do in this paper, it becomes transparent that there is a parallel theory of oscillator representations (which are infinite-dimensional) for the symplectic Lie algebra, so we develop them both simultaneously. Furthermore, it will follow that the correct version of a symplectic Lie algebra on an odd-dimensional space to use here is the orthosymplectic Lie superalgebra of a symplectic space with $1$ odd variable (the finite-dimensional representation theory of this Lie superalgebra is semisimple, unlike the stabilizer Lie algebra of a maximal rank alternating form on a $2n+1$ dimensional space, and also unlike most other finite-dimensional classical Lie superalgebras).

As in \cite{infrank}, the first step is to construct a category of ``infinite rank'' spinor (or oscillator) representations and evaluation functors to the representation categories in finite rank. The idea of working with tensor categories of representations of infinite rank Lie algebras previously appeared in \cite{penkovserganova} and \cite{penkovstyrkas}. We give three different models for these categories: one in terms of multilinear algebraic constructions, one in terms of diagram categories, and one in terms of finite length modules over a twisted Lie (super)algebra. The latter interpretation allows us to use intuition from commutative algebra and there is an advantage in using modules which are not finite length (see \cite{deg2tca} for this line of inquiry in the case of orthogonal and symplectic groups). In particular, we use this interpretation to construct explicit projective resolutions for the simple objects in the infinite rank category. 

We then calculate the homology of these complexes after applying the evaluation functor. As in \cite{lwood}, this calculation proceeds by interpreting the homology groups as Tor groups for certain modules over a Lie algebra which is the nilpotent radical in a larger classical Lie (super)algebra (coming from a Howe dual pair), and this Tor is calculated using the ``geometric technique'' explicated in \cite{weyman}. One difference is that in \cite{lwood}, the nilpotent radical is always abelian, so that one is strictly dealing with commutative rings. As a side product, we find analogues of determinantal ideals in the universal enveloping algebras of free $2$-step nilpotent Lie (super)algebras and calculate their minimal free resolutions in \S\ref{sec:determinantal}.

Finally, some of the results in this paper are foreshadowed by existing combinatorial results. The universal character ring for symplectic and orthogonal groups developed in \cite{koiketerada} provide alternative bases for the ring of symmetric functions which behave very similarly to the Schur functions. In fact, there is a third such basis with the same multiplication structure constants (see \cite[Theorem 6]{shimozono} which summarizes \cite{kleber}). This paper shows that this basis comes from an analogous theory of universal character ring for spinor representations (or oscillator representations), see also \cite{koike-diff} for an elaboration on the combinatorics.

\subsection{Background and notation}

We always work over the complex numbers $\bC$. $\cV$ denotes the category of polynomial representations of $\GL_\infty$, see \cite[\S\S 5,6]{expos} for its basic properties and other equivalent categories. We will freely make use of Schur functors $\bS_\lambda$; for background, the reader might consult \cite[\S 6]{fultonharris}. The symmetric group on $n$ letters is denoted $S_n$.

We will also make use of standard partition notation, see \cite[\S 1]{expos} for basic definitions. The transpose of a partition $\lambda$ is denoted $\lambda^\dagger$. Let $Q_{-1}$ be the set of partitions with the following inductive definition. The empty partition belongs to $Q_{-1}$.  A non-empty partition $\mu$ belongs to $Q_{-1}$ if and only if the number of rows in $\mu$ is one more than the number of columns, i.e., $\ell(\mu)=\mu_1+1$, and the partition obtained by deleting the first row and column of $\mu$, i.e., $(\mu_2-1, \ldots, \mu_{\ell(\mu)}-1)$, belongs to $Q_{-1}$. The first few partitions in $Q_{-1}$ are $0$, $(1,1)$, $(2,1,1)$, $(2,2,2)$. Define $Q_1 = \{\lambda \mid \lambda^\dagger \in Q_{-1}\}$. We write $Q_{-1}(2i)$ for the set of $\lambda \in Q_{-1}$ with $|\lambda|=2i$, and similarly we define $Q_1(2i)$. 

The significance of these sets are the following decompositions (see \cite[\S I.A.7, Ex.~4,5]{macdonald}):
\begin{align} \label{eqn:wedge-plethysm}
\lw^{i}(\Sym^2(E)) = \bigoplus_{\substack{\mu \in Q_1(2i)}} \bS_{\mu}(E), \qquad 
\lw^{i}(\lw^2(E)) = \bigoplus_{\substack{\mu \in Q_{-1}(2i)}} \bS_{\mu}(E).
\end{align}

\begin{remark} \label{rmk:Q1-num}
Each partition in $Q_{-1}$ is a union of certain hook partitions which are determined by their first part $a$. So alternatively, they can be specified by a strictly decreasing sequence $a_1 > a_2 > \cdots > a_r$. So the total number of partitions $\mu \in Q_{-1}$ with $\mu_1 \le n$ is $2^n$. In particular, the total number of partitions $\mu \in Q_1$ with $\ell(\mu) \le n$ is also $2^n$.
\end{remark}

Given a partition $\lambda$, let $\rank(\lambda) = \max \{ i \mid \lambda_i \ge i\}$ denote the size of its main diagonal. 

Given two vector spaces $E, F$, let $E \uplus F$ denote the $\bZ/2$-graded vector space with even part $E$ and odd part $F$.

Finally, $\delta$ is usually used to denote the highest weight of the spinor representation (so implicitly $\delta$ depends on the rank of the group, but will usually be clear from the context). Similarly, $\eta$ will be used to denote the highest weight of the oscillator representation.

\subsection{Acknowledgements}

We thank Robert Laudone for carefully going through the formulas in \S\ref{sec:spinor} and pointing out some scalar errors that were originally present.

\section{Infinite rank spinor representations} \label{sec:spinor}

\subsection{Basic definitions}

Let $\bW=\bC^{\infty}=\bigcup_{n \ge 1} \bC^n$, let $\bW_* = \bigcup_{n \ge 1} (\bC^n)^*$ be its restricted dual, and put 
\[
\ol{\bV}=\bW \oplus \bW_*, \qquad \bV = \ol{\bV} \oplus \bC.
\]
We let $\be$ be a basis vector for the one dimensional space $\bC$ in $\bV$. We put an orthogonal form $\omega$ on $\ol{\bV}$ by 
\[
\omega((v, f), (v', f')) = f'(v) + f(v').
\]
We extend this to an orthogonal form, also called $\omega$, to $\bV$ by setting $\omega(\be,\be) = 1$ and $\omega(\be,v) = 0$ for all $v \in \ol{\bV}$. Put
\begin{displaymath}
\fg = \lw^2(\bV)
= \lw^2(\bW) \oplus (\bW \otimes \be) \oplus (\bW \otimes \bW_*) \oplus (\bW_* \otimes \be) \oplus \lw^2(\bW_*).
\end{displaymath}
We regard $\fg$ as $\bZ$-graded, with $\bW$ of degree~1, $\bW_*$ of degree~$-1$, and $\be$ of degree~0. We define elements of $\fg$ as follows.
\begin{itemize}
\item For $v, w \in \bW$ we let $x_{v,w}=v \wedge w$ and $x_v = v \otimes \be$.
\item For $v \in \bW$ and $\phi \in \bW_*$ we let $h_{v,\phi}=v \otimes \phi$.
\item For $\phi, \psi \in \bW_*$ we let $y_{\phi,\psi}=\phi \wedge \psi$ and $y_\phi = \phi \otimes \be$.
\end{itemize}
Define a map $\fg \to \fgl(\bV)$ as follows. Suppose $u \in \bW \subset \bV$. Then
\begin{displaymath}
x_{v,w} u=0, \qquad x_v u = 0, \qquad h_{v,\phi} u = \phi(u) v, \qquad y_\phi u = \phi(u) \be, \qquad y_{\phi,\psi} u = \psi(u) \phi - \phi(u) \psi.
\end{displaymath}
We define the action on $\eta \in \bW_*$ in an analogous manner:
\begin{displaymath}
x_{v,w} \eta= \eta(w)v - \eta(v)w, \qquad x_v \eta = -\eta(v) \be, \qquad h_{v,\phi} \eta = -\eta(v) \phi, \qquad y_\phi \eta = 0, \qquad y_{\phi,\psi} \eta = 0.
\end{displaymath}
We also put
\begin{displaymath}
x_{v,w} \be=0, \qquad x_v \be = v, \qquad h_{v,\phi} \be = 0, \qquad y_\phi \be = -\phi, \qquad y_{\phi,\psi} \be=0.
\end{displaymath}
Then $\fg$ is closed under the Lie bracket on $\fgl(\bV)$ and so is a Lie algebra. It preserves the orthogonal form on $\bV$, and can be rightfully called $\fso(2\infty+1)$. We call $\bV$ the {\bf standard representation} of $\fg$. A representation of $\fg$ is {\bf algebraic} if it appears as a subquotient of a finite direct sum of tensor powers of the standard representation. The category of algebraic representations is denoted $\Rep(\fg)$ and is studied in \cite[\S 4]{infrank}. It is a symmetric monoidal abelian category.

Now let $\Delta^n=\lw^n{\bW}$, and let $\Delta=\bigoplus_{n \ge 0} \Delta^n$ be the exterior algebra on $\bW$. For $v \in \bW$, we let $X_v$ be the operator on $\Delta$ given by 
\[
X_v(w)=v \wedge w.
\]
The operators $X_v$ and $X_w$ supercommute, that is, $X_vX_w+X_wX_v=0$. For $\phi \in \bW_*$, we let $Y_{\phi}$ be the operator on $\Delta$ given by
\begin{displaymath}
Y_{\phi}(v_1 \wedge \cdots \wedge v_n)=\sum_{i=1}^n (-1)^{i-1} \phi(v_i) v_1 \wedge \cdots \wedge \wh{v}_i \wedge \cdots \wedge v_n.
\end{displaymath}
The operators $Y_{\phi}$ and $Y_{\psi}$ also supercommute and $X_vY_{\phi} + Y_{\phi}X_v = \phi(v)$. We let 
\[
H_{v,\phi}=X_v Y_{\phi};
\]
this is the usual action of the element $v \phi \in \mf{gl}(\bW)$ on $\Delta$. Finally, define $D$ by 
\[
D(v_1 \wedge \cdots \wedge v_n) = (-1)^n v_1 \wedge \cdots \wedge v_n.
\]
Then $D$ supercommutes with all $X_v$ and all $Y_\phi$. 

Define a representation $\rho$ of $\fg$ on $\Delta$ as follows:
\begin{align*}
&\rho(x_{v,w}) = X_v X_w, \qquad
\rho(x_v) = \tfrac{1}{\sqrt{2}} X_v D, \qquad
\rho(h_{v,\phi})=H_{v,\phi}-\tfrac{1}{2} \phi(v), \\
&\rho(y_\phi) = \tfrac{1}{\sqrt{2}}D Y_\phi, \qquad
\rho(y_{\phi,\psi})=Y_{\phi} Y_{\psi}.
\end{align*}
We leave it to the reader to verify that this is a well-defined representation. Alternatively, one can use \cite[\S 20.1]{fultonharris}, specifically the map in \cite[(20.6)]{fultonharris} (here we make the identifications $\bW = \bW \otimes \be \subset \bigwedge^2 \bV$ and $\bW_* = \bW_* \otimes \be \subset \bigwedge^2 \bV$). This is the {\bf spinor representation} of $\fg$.

The map $\bV \to \End(\Delta)$ defined by $v \mapsto X_v$ and $\phi \mapsto Y_{\phi}$ and $\be \mapsto \frac{1}{\sqrt{2}} D$ is a map of $\fg$-representations (see \cite[Proof of Lemma 20.16]{fultonharris}). It follows that the map 
\begin{equation}
\label{eq:Vdelta}
\bV \otimes \Delta \to \Delta
\end{equation}
given by $v \otimes x \mapsto X_v x$ and $\phi \otimes x \mapsto Y_{\phi} x$ and $\be \otimes x \mapsto \frac{1}{\sqrt{2}} Dx$ is a map of $\fg$-representations.

A representation of $\fg$ is {\bf spin-algebraic} if it appears as a subquotient of a finite direct sum of representations of the form 
\[
T^n = \bV^{\otimes n} \otimes \Delta.
\]
Let $\Rep^{\spin}(\fg)$ denote the category of spin-algebraic representations of $\fg$. It is an abelian category, and is naturally a module over the tensor category $\Rep(\fg)$, i.e., we have a bifunctor given by tensor product
\begin{align} \label{eqn:fg-tensor}
\otimes \colon \Rep(\fg) \times \Rep^\spin(\fg) \to \Rep^\spin(\fg).
\end{align}

\subsection{Weyl's construction} \label{sec:weyl-spin}

Pick a basis $e_1, e_2, \dots$ for $\bW$ and let $e_1^*, e_2^*, \dots$ be the dual basis of $\bW_*$. Also set $e_{-i} = e_i^*$ and put $e_0=\be$. Let $\ft$ be the diagonal torus in $\mf{gl}(\bW) \subset \fg$, which is spanned by the $h_i=h_{e_i,e_i^*}$. Let $X$ be the lattice of integral characters of $\ft$ of finite support: $X$ consists of all maps $\ft \to \bC$ which send each $h_i$ to an integer, and all but finitely many $h_i$ to~0. We identify $X$ with the set of sequences $(a_1, a_2, \ldots)$ of integers with $a_i=0$ for $i \gg 0$. The {\bf magnitude} of an element $(a_i)$ of $X$ is $\sum_i |a_i|$. The magnitude of an algebraic representation of $\mf{gl}(\bW)$ or $\fg$ is the maximum magnitude of a weight in it. 

Let $X'$ be the set of characters $\ft \to \bC$ such that each $h_i$ is sent to a half-integer (and not an integer) and $h_i$ is sent to $-\tfrac{1}{2}$ for $i \gg 0$. Then $X'$ is isomorphic to the set of sequences $(a_1, a_2, \ldots)$ where $a_i$ is a half-integer and $a_i=-\tfrac{1}{2}$ for $i \gg 0$. The set $X'$ is not a group under addition, but one can add elements of $X$ and $X'$ and get an element of $X'$; in fact, $X'$ is an $X$-torsor. Let $\lambda_0 \in X'$ be the sequence $(a_i)$ with $a_i = -\tfrac{1}{2}$ for all $i$. We define the {\bf magnitude} of an element $\lambda \in X'$ to be the magnitude of $\lambda - \lambda_0$. 

Let $\fp \subset \fg$ be the parabolic subalgebra spanned by the $h$'s and $y$'s and let $\fn$ be the nilpotent subalgebra spanned by the $y$'s. Let also $\fb \subset \fg$ be the Borel subalgebra spanned by the $y$'s and $h_{e_i, e_j^*}$ where $j < i$. Call a non-zero vector of a $\fg$-representation a highest weight vector if it is annihilated by $\fb$.

The space $T^n$ inherits a grading from the grading of $\Delta$.

\begin{proposition} \label{prop:spin-mag}
Let $M$ be a non-zero $\fg$-submodule of $T^n$.
\begin{enumerate}[\rm (a)]
\item $M \cap \bV^{\otimes n} \otimes \Delta^0 \ne 0$.
\item If $M$ is simple, then it contains a unique, up to scalar, highest weight vector. The weight of this vector has magnitude $n$. 
\end{enumerate}
\end{proposition}

\begin{proof}
(a) The grading on $T^n$ makes it into a filtered $\fp$-module:
\[
F^0 T^n \subset F^1 T^n \subset F^2 T^n \subset F^3 T^n \subset \cdots
\]
where $F^i T^n = \sum_{j \le i} \bV^{\otimes n} \otimes \bigwedge^j \bW$. The associated graded module is $\bV^{\otimes n} \otimes \Delta$, with $\fn$ acting trivially on $\Delta$ and in the usual way on $\bV$. 

We now claim that $M \cap F^r T^n \ne 0$, for each $r \ge 0$. This is clear for $r \gg 0$. Suppose that $M \cap F^{r+1} T^n \ne 0$; we will show that $M \cap F^r T^n \ne 0$. Since $\bV^{\otimes n}$ contains non-zero vectors annihilated by $\fn$, the above discussion shows that we can pick non-zero $v \in M \cap F^{r+1} T^n$ which is killed by $\fn$ in the associated graded. If $v \in M \cap F^r T^n$, there is nothing to do, so assume otherwise and write 
\[
v = \sum c_{I,J} (e_{i_1} \otimes \cdots \otimes e_{i_n}) \otimes (e_{j_1} \wedge \cdots \wedge e_{j_{r+1}})
\]
where $i_1, \dots, i_n$ are integers and $0 < j_1 < \cdots < j_{r+1}$. For any $X \in \fn$, we have $Xv \in F^r T^n$, and we can pick $X$ so that $Xv$ is non-zero: let $X = y_{e_j^*}$ where $j \in J$ and $c_{I,J} \ne 0$. So $M \cap F^r T^n \ne 0$ and so by induction, $M \cap F^0 T^n \ne 0$. This completes the proof.

(b) For each $n$, the span of the tensors built out of $e_1, \dots, e_n, e_1^*, \dots, e_n^*$ is a simple submodule of $M$ for the subalgebra $\fso(2n+1)$ (assuming it is non-zero), so the uniqueness of a vector annihilated by $\fb \cap \fso(2n+1)$ (i.e., highest weight vector) follows by generalities of representation theory. Furthermore, this vector will be the same for each $n$, provided that $n$ is large enough for this submodule to be non-zero. Furthermore, from the discussion above, this vector belongs to $\bV^{\otimes n} \otimes \Delta^0$. The action of $\fgl(\bW)$ on $\bV^{\otimes n} \otimes \Delta^0$ is the usual action tensored with $-\tfrac{1}{2}$ times the trace character. Since $\fb \cap \fgl(\bW)$ is a Borel subalgebra of $\fgl(\bW)$, the weight of this vector has magnitude $n$ by \cite[Proposition 4.1.8]{infrank}.
\end{proof}

Let $t_i \colon T^n \to T^{n-1}$ be the map given by applying the map $\bV \otimes \Delta \to \Delta$ from \eqref{eq:Vdelta} to the $i$th factor. An important property of these maps is that $t_i t_j+t_j t_i$ is induced from the map $\bV^{\otimes n} \to \bV^{\otimes (n-2)}$ given by pairing the $i$th and $j$th factors. (To prove this, it suffices to treat the $n=2$ case, where it follows from an easy calculation using the formulas defining \eqref{eq:Vdelta}.) Define $T^{[n]}$ to be the intersection of the kernels of the $t_i$. This is stable under the action of $S_n \times \fg$. Finally, define
\begin{displaymath}
\Delta_{\lambda}=\Hom_{S_n}(\bM_{\lambda}, T^{[n]})
\end{displaymath}
where $\bM_\lambda$ is the irreducible representation of $S_n$ indexed by $\lambda$.

To analyze $\Delta_\lambda$, we first need to understand the finite-dimensional case. So let $V$ be an $2n+1$-dimensional orthogonal space and define $\Delta^{(n)}_\lambda$ in an analogous way.

\begin{lemma} \label{lem:odd-spinor-ker}
Let $\delta$ be the highest weight for $\Delta^{(n)}$. If $\ell(\lambda) \le n$ then $\Delta^{(n)}_\lambda$ is an irreducible representation of $\Spin(V)$ of highest weight $\lambda + \delta$.
\end{lemma}

\begin{proof}
This follows from the presentation of the module $M$ in \cite[Proposition 4.8]{exceptional}. More specifically, let $E$ be an $n$-dimensional vector space and set $A = \Sym(E \otimes V)$. We have a degree $1$ $A$-linear $\GL(E) \times \Spin(V)$-equivariant map
\[
E \otimes \Delta^{(n)} \otimes A \to \Delta^{(n)} \otimes A
\]
which is induced by inclusion $E \otimes \Delta^{(n)} \to \Delta^{(n)} \otimes E \otimes \Delta^{(n)}$ which is $t \colon E \to \Delta^{(n)} \otimes E$ tensored with the identity on $\Delta^{(n)}$. The quotient of this map is the module $M$ which has a $\GL(E) \times \Spin(V)$-equivariant decomposition
\[
M = \bigoplus_{\substack{\lambda\\ \ell(\lambda) \le n}} \bS_\lambda E \otimes \Delta_\lambda^{(n)}.
\]
The Cauchy identity says that $A = \bigoplus_\lambda \bS_\lambda E \otimes \bS_\lambda V$. Taking the $\bS_\lambda E$-isotypic component of the presentation of $M$ then gives us 
\[
\Delta^{(n)} \otimes \bS_{\lambda/1} V \to \Delta^{(n)} \otimes \bS_\lambda V \to \Delta^{(n)}_\lambda \to 0
\]
where $\bS_{\lambda/1} V$ denotes the skew Schur functor. So $\Delta^{(n)}_\lambda$ is the quotient of $\bS_\lambda(V) \otimes \Delta^{(n)}$ by the image of the maps $t_i$. To get it as the kernel of the maps $t_i$, we can translate by taking duals. 
\end{proof}

\begin{proposition} \label{prop:spin-simples}
As $\lambda$ ranges over all partitions, the representations $\Delta_{\lambda}$ are a complete irredundant set of simple objects of $\Rep^{\spin}(\fg)$.
\end{proposition}

\addtocounter{equation}{-1}
\begin{subequations}
\begin{proof}
Let $\bW^{(n)}$ be the span of $e_1, \ldots, e_n$, and define $\bW_*^{(n)}$ similarly. As above, we have a representation $\Delta_{\lambda}^{(n)}$ of $\fg^{(n)} \cong \fso(2n+1)$, which is simple by Lemma~\ref{lem:odd-spinor-ker}. Since $\Delta_{\lambda} = \bigcup_{n \ge 1} \Delta_{\lambda}^{(n)}$ and $\fg=\bigcup_{n \ge 1} \fg^{(n)}$, it follows that $\Delta_{\lambda}$ is irreducible for $\fg$, and hence is a simple object of $\Rep^\spin(\fg)$. By the definition of $\Delta_{\lambda}$ and semisimplicity of $S_n$, we have a decomposition
\begin{equation} \label{tdecomp}
T^{[n]} = \bigoplus_{|\lambda|=n} \bM_\lambda \otimes \Delta_\lambda
\end{equation}
as a representation of $S_n \times \fg$, and so every constituent of $T^{[n]}$ (as a $\fg$-representation) is isomorphic to some $\Delta_{\lambda}$. We have an exact sequence
\begin{equation} \label{tseq}
0 \to T^{[n]} \to T^n \to (T^{n-1})^{\oplus n},
\end{equation}
and so (by induction) every simple constituent of $T^n$ is one of $T^{[m]}$ for some $m \le n$, and thus of the form $\Delta_\lambda$ with $|\lambda| \le n$. Since every simple object of $\Rep^{\spin}(\fg)$ is a constituent of some $T^n$, it follows that every simple object is of the form $\Delta_\lambda$ for some $\lambda$.

Under the action of the diagonal torus, the characters of $\bS_\lambda(\bV) \otimes \Delta$ are linearly independent, and in the Grothendieck group of $\Rep^\spin(\fg)$, we have 
\[
[\bS_\lambda(\bV) \otimes \Delta] = [\Delta_\lambda] + \sum_{|\mu| < |\lambda|} c^\lambda_\mu [\Delta_\mu]
\]
for some coefficients $c^\lambda_\mu$. So the characters of the $\Delta_\lambda$ are linearly independent, and hence they are pairwise non-isomorphic.
\end{proof}
\end{subequations}

\begin{corollary}
Every object of $\Rep^\spin(\fg)$ has finite length.
\end{corollary}

\begin{proof}
The simplicity of $\Delta_{\lambda}$ and \eqref{tdecomp} shows that $T^{[n]}$ is finite length. The sequence \eqref{tseq} then inductively shows that $T^n$ is finite length. Since every object is a quotient of a finite sum of $T^n$'s, the result follows.
\end{proof}

\begin{proposition} \label{prop:spin-hom-trace}
We have the following:
\begin{displaymath}
\Hom_{\fg}(\Delta_\lambda, T^n) = \begin{cases}
\bM_{\lambda} & \textrm{if $n=\vert \lambda \vert$} \\
0 & \textrm{otherwise}
\end{cases}.
\end{displaymath}
\end{proposition}

\begin{proof}
By Proposition~\ref{prop:spin-simples}, every constituent of $T^n$ is of the form $\Delta_{\mu}$ with $\vert \mu \vert \le n$. Thus if $\vert \lambda \vert>n$, there are no non-zero maps $\Delta_{\lambda} \to T^n$. Now suppose that $\vert \lambda \vert<n$. Then the image of a non-zero map $\Delta_{\lambda} \to T^n$ has non-zero intersection with $\bV^{\otimes n} \otimes \Delta^0$ by Proposition~\ref{prop:spin-mag}. Furthermore, the weight of the highest weight vector of $\Delta_\lambda$ with respect to $\fb$ has magnitude $|\lambda| < n$ which contradicts Proposition~\ref{prop:spin-mag}(b). Thus, again, there are no non-zero maps $\Delta_{\lambda} \to T^n$. Finally, suppose $n=\vert \lambda \vert$. Using the sequence \eqref{tseq} we see that $\Hom_{\fg}(\Delta_{\lambda}, T^n)=\Hom_{\fg}(\Delta_{\lambda}, T^{[n]})$, and the result follows from \eqref{tdecomp}.
\end{proof}

\subsection{Diagram category}

A {\bf spin-Brauer diagram} between a set $L$ and $L'$ is a triple $(U, \Gamma, f)$ where 
\begin{itemize}
\item $U$ is a subset of $L$ equipped with a total order,
\item $\Gamma$ is a (partial) matching on $L \setminus U$,
\item $f$ is a bijection $L \setminus (U \cup V(\Gamma)) \to L'$. Here $V(\Gamma)$ is the vertex set of $\Gamma$.
\end{itemize}
We just write $\Gamma$ for such diagrams, and think of the vertices in $U$ as circled. Here we have drawn an example of a spin-Brauer diagram:
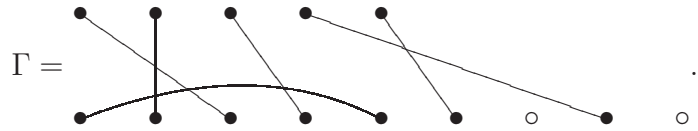
\[
\Gamma = \begin{xy}
(-28, 7)*{}="A1"; (-18, 7)*{}="B1"; (-8, 7)*{}="C1"; (2, 7)*{}="D1"; (12, 7)*{}="E1"; 
(-28, -7)*{}="A2"; (-18, -7)*{}="B2"; (-8, -7)*{}="C2"; (2, -7)*{}="D2"; (12, -7)*{}="E2"; (22, -7)*{}="F2"; (32, -7)*{}="G2"; (42, -7)*{}="H2"; (52, -7)*{}="I2";
"A1"*{\bullet}; "B1"*{\bullet}; "C1"*{\bullet}; "D1"*{\bullet}; "E1"*{\bullet}; 
"A2"*{\bullet}; "B2"*{\bullet}; "C2"*{\bullet}; "D2"*{\bullet}; "E2"*{\bullet}; "F2"*{\bullet}; "G2"*{\circ}; "H2"*{\bullet}; "I2"*{\circ};
"A1"; "C2"; **\dir{-};
"B1"; "B2"; **\dir{-};
"C1"; "D2"; **\dir{-};
"E1"; "F2"; **\dir{-};
"A2"; "E2"; **\crv{(-5, 2)};
"H2"; "D1"; **\dir{-};
\end{xy}.
\]
We have not indicated the total ordering on the elements of $U$ in this drawing. Suppose $\Gamma$ is a spin-Brauer diagram and $\Gamma'$ is obtained by switching the two consecutive elements $i, j \in U$ in the total order. Let $\Gamma''$ be the spin-Brauer diagram obtained by removing $i$ and $j$ from $U$ in $\Gamma$, and placing an edge between them. The {\bf Clifford relation} is $\Gamma+\Gamma'=\Gamma''$. Similar types of diagrams were considered by Koike \cite{koike-spin}.

The {\bf downwards spin category}, denoted $\dspin$, is the following $\bC$-linear category. The objects are finite sets $L$ and $\Hom_\dspin(L,L')$ is the quotient of the vector space spanned by the spin-Brauer diagrams by the Clifford relations. Given $(U, \Gamma, f) \in \Hom_\dspin(L,L')$ and $(U', \Gamma', f') \in \Hom_\dspin(L', L'')$, define their composition to be $(U \cup f^{-1}(U'), \Gamma \cup f^{-1}(\Gamma'), g)$ where $g$ is the restriction of $f'f$ to
$L \setminus (U \cup f^{-1}(U') \cup V(\Gamma) \cup V(f^{-1}(\Gamma')))$. To give a sense of this with our example above, if 
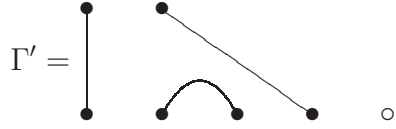
\[
\Gamma' = \begin{xy}
(-28, 7)*{}="A1"; (-18, 7)*{}="B1"; 
(-28, -7)*{}="A2"; (-18, -7)*{}="B2"; (-8, -7)*{}="C2"; (2, -7)*{}="D2"; (12, -7)*{}="E2"; 
"A1"*{\bullet}; "B1"*{\bullet}; 
"A2"*{\bullet}; "B2"*{\bullet}; "C2"*{\bullet}; "D2"*{\bullet}; "E2"*{\circ}; 
"A1"; "A2"; **\dir{-};
"B1"; "D2"; **\dir{-};
"B2"; "C2"; **\crv{(-13,2)};
\end{xy}
\]
then the composition $\Gamma' \Gamma$ is obtained by first concatenating the diagrams
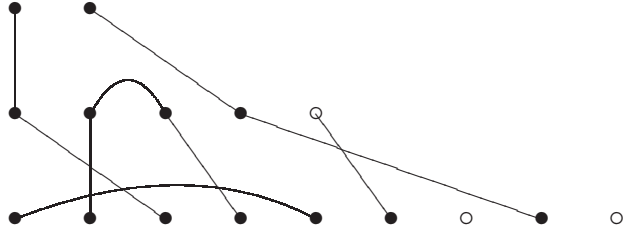
\[
\begin{xy}
(-28, 21)*{}="A0"; (-18, 21)*{}="B0"; 
(-28, 7)*{}="A1"; (-18, 7)*{}="B1"; (-8, 7)*{}="C1"; (2, 7)*{}="D1"; (12, 7)*{}="E1"; 
(-28, -7)*{}="A2"; (-18, -7)*{}="B2"; (-8, -7)*{}="C2"; (2, -7)*{}="D2"; (12, -7)*{}="E2"; (22, -7)*{}="F2"; (32, -7)*{}="G2"; (42, -7)*{}="H2"; (52, -7)*{}="I2";
(-28, -7)*{}="A2"; (-18, -7)*{}="B2"; (-8, -7)*{}="C2"; (2, -7)*{}="D2"; (12, -7)*{}="E2"; 
"A0"*{\bullet}; "B0"*{\bullet}; 
"A0"; "A1"; **\dir{-};
"B0"; "D1"; **\dir{-};
"B1"; "C1"; **\crv{(-13,16)};
"A1"*{\bullet}; "B1"*{\bullet}; "C1"*{\bullet}; "D1"*{\bullet}; "E1"*{\circ}; 
"A2"*{\bullet}; "B2"*{\bullet}; "C2"*{\bullet}; "D2"*{\bullet}; "E2"*{\bullet}; "F2"*{\bullet}; "G2"*{\circ}; "H2"*{\bullet}; "I2"*{\circ};
"A1"; "C2"; **\dir{-};
"B1"; "B2"; **\dir{-};
"C1"; "D2"; **\dir{-};
"E1"; "F2"; **\dir{-};
"A2"; "E2"; **\crv{(-5, 2)};
"H2"; "D1"; **\dir{-};
\end{xy}
\]
and then simplifying to get the diagram
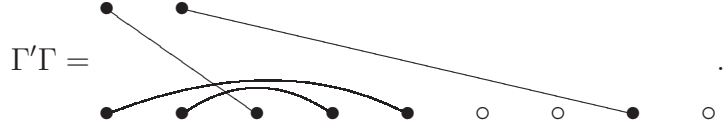
\[
\Gamma' \Gamma = \begin{xy}
(-28, 7)*{}="A1"; (-18, 7)*{}="B1"; 
(-28, -7)*{}="A2"; (-18, -7)*{}="B2"; (-8, -7)*{}="C2"; (2, -7)*{}="D2"; (12, -7)*{}="E2"; (22, -7)*{}="F2"; (32, -7)*{}="G2"; (42, -7)*{}="H2"; (52, -7)*{}="I2";
"A1"*{\bullet}; "B1"*{\bullet}; 
"A2"*{\bullet}; "B2"*{\bullet}; "C2"*{\bullet}; "D2"*{\bullet}; "E2"*{\bullet}; "F2"*{\circ}; "G2"*{\circ}; "H2"*{\bullet}; "I2"*{\circ};
"A1"; "C2"; **\dir{-};
"A2"; "E2"; **\crv{(-5, 2)};
"B2"; "D2"; **\crv{(-8, 0)};
"H2"; "B1"; **\dir{-};
\end{xy}.
\]
The {\bf upwards spin category} is $\uspin = \dspin^{\rm op}$. A {\bf representation} of $\dspin$ is a $\bC$-linear functor $F \colon \dspin \to \Vec$. Let $\Mod_\dspin$ (resp., $\Mod_\dspin^\rf$) denote the abelian category of all (resp., finite length) representations of $\dspin$. 

For a finite set $L$, put $\cK_L=\bV^{\otimes L} \otimes \Delta$. Given a spin-Brauer diagram $(U,\Gamma,f)$ between $L$ and $L'$, consider the sequence of maps
\begin{displaymath}
\cK_L \to \cK_{L \setminus U} \to \cK_{L \setminus (U \cup V(\Gamma))} \to \cK_{L'},
\end{displaymath}
defined as follows. The first map applies the map $\bV \otimes \Delta \to \Delta$ to the tensor factors indexed by $U$, in order (according to the total order on $U$). The second map applies the pairing $\bV^{\otimes 2} \to \bC$ to the tensor factors indexed by the edges of $\Gamma$. The third map is induced by the bijection $f$. One readily verifies that the above definition gives $\cK$ the structure of a $\bC$-linear functor $\dspin \to \Rep^{\spin}(\fg)$.
 
In \cite[(4.2.5)]{infrank}, we defined $\db$ (resp., $\ub$), which in our current notation is the subcategory of $\dspin$ (resp., $\uspin$) where we only use spin-Brauer diagrams with $U = \emptyset$. This has a symmetric monoidal structure using disjoint union $\amalg$. Clearly, $\amalg$ extends to a bifunctor $\db \times \dspin \to \dspin$. Using \cite[(2.1.14)]{infrank}, we get a convolution tensor product
\begin{equation} \label{eqn:amalg-spin0}
\amalg_{\#} \colon \Mod_\db^\rf \times \Mod_\dspin^\rf \to \Mod_\dspin^\rf,
\end{equation}

\begin{theorem} \label{thm:kernel-spin}
The kernel $\cK$ induces mutually quasi-inverse anti-equivalences between $\Mod_\dspin^\rf$ and $\Rep^{\spin}(\fg)$. Moreover, under this equivalence, the tensor product in \eqref{eqn:amalg-spin0} corresponds to the tensor product in \eqref{eqn:fg-tensor}.
\end{theorem}

\begin{proof}
We apply \cite[Theorem 2.1.11]{infrank}, so we need to check its two hypotheses. Write $\cK_n$ for $\cK_n$ evaluated on the set $\{1,\dots,n\}$. The first hypothesis is that $\hom_{S_n}(\bM_\lambda, \cK_n)$ is irreducible; this follows from Proposition~\ref{prop:spin-simples}. The second hypothesis is a converse statement that $\hom_{\Rep^\spin(\fg)}(\Delta_\lambda, \cK_n)$ is a nonzero irreducible representation of $S_n$; this follows from Proposition~\ref{prop:spin-hom-trace}. The compatibility with tensor products is \cite[Proposition 2.1.16]{infrank}.
\end{proof}

\begin{proposition} \label{prop:spin-inj}
For every partition $\lambda$, the representation $\bS_\lambda(\bV) \otimes \Delta$ is an injective object of $\Rep^\spin(\fg)$, and is the injective envelope of the simple $\Delta_\lambda$.
\end{proposition}

\begin{proof}
We can use an argument similar to the proof of \cite[Proposition 3.2.14]{infrank}.
\end{proof}

\begin{corollary}
For every $n \ge 0$, $T^n$ is an injective object of $\Rep^\spin(\fg)$.
\end{corollary}

\begin{proof}
We have a decomposition $T^n = \bigoplus_{\lambda,\ |\lambda|=n} (\bS_\lambda(\bV) \otimes \Delta)^{\oplus \dim \bM_\lambda}$.
\end{proof}

\subsection{Universal property} \label{ss:spinuniv}

Consider pairs of categories $(\cA, \cB)$ where $\cA$ is a symmetric monoidal abelian category and $\cB$ is an $\cA$-module, i.e., we have a biadditive bifunctor $\otimes \colon \cA \times \cB \to \cB$ equipped with the appropriate extra structure (e.g., an associator).

Suppose that $(\cA', \cB')$ is a second pair. By a functor $F \colon (\cA, \cB) \to (\cA', \cB')$ we mean a symmetric tensor functor $F_1 \colon \cA \to \cA'$ and a functor $F_2 \colon \cB \to \cB'$ of $\cA$-modules, both of which are additive. We say that $F$ is left-exact if both $F_1$ and $F_2$ are. Let ${\rm LEx}^{\otimes}( (\cA, \cB), (\cA', \cB'))$ denote the category whose objects are left-exact functors $(\cA, \cB) \to (\cA', \cB')$ and whose morphisms are natural transformations compatible with the extra structure.

Suppose that $A \in \cA$ and $\omega \colon \Sym^2(\cA) \to \bC$ is a symmetric bilinear form on $A$. If $\cA$ has infinite direct sums, we can form the Clifford algebra $\Cl(A)$ as the usual quotient of the tensor algebra on $A$ (see \cite[\S 20.1]{fultonharris}). We can then speak of $\Cl(A)$-modules in $\cB$. Even if $\cA$ does not have infinite direct sums, we can still define the notion of a $\Cl(A)$-module in $\cB$: it is an object $B$ of $\cB$ equipped with a morphism $t \colon A \otimes B \to B$ such that the two maps
\begin{displaymath}
f, g \colon A \otimes A \otimes B \to B 
\end{displaymath}
given by
\begin{displaymath}
f(x \otimes y \otimes m)=t(x \otimes t(y \otimes m))+t(y \otimes t(x \otimes m)), \qquad
g(x \otimes y \otimes m) = \omega(x, y) m
\end{displaymath}
agree. We let $T(\cA, \cB)$ be the category whose objects are tuples $(A, \omega, B, t)$ as above. We write $(A, B)$ for an object of $T(\cA, \cB)$ when there is no danger of confusion.

Given $(A, B) \in T(\cA, \cB)$, define $\cK(A) \colon \db \to \cA$ by $L \mapsto A^{\otimes L}$ and similarly, define $\cK(B) \colon \dspin \to \cB$ by $L \mapsto A^{\otimes L} \otimes B$. For an object $M$ of $\Mod_\ub^\rf$ and an object $N$ of $\Mod_\uspin^\rf$, define (using \cite[(2.1.9)]{infrank}) objects of $\cA$ and $\cB$, respectively,
\begin{align*}
S_M(A) = M \otimes^\db \cK(A),\qquad S_N(B) = N \otimes^\dspin \cK(B).
\end{align*}
Then $(M,N) \mapsto (S_M(A), S_N(B))$ is a left-exact tensor functor $(\Mod_\ub^\rf, \Mod_\uspin^\rf) \to (\cA, \cB)$.

\begin{theorem} \label{thm:spin-univ}
Giving a left-exact tensor functor
\begin{displaymath}
(\Rep(\fg), \Rep^\spin(\fg)) \to (\cA, \cB)
\end{displaymath}
is the same as giving an object of $T(\cA, \cB)$.  More precisely, let $\bM$ be the object of $\Mod_\ub^\rf$ corresponding to $\bV$ in $\Rep(\fg)$ and let $\bN$ be the object of $\Mod_\uspin^\rf$ corresponding to $\Delta$ in $\Rep^\spin(\fg)$. Then the functors
\[
{\rm LEx}^{\otimes}( (\Mod_\ub^\rf, \Mod_\uspin^\rf), (\cA, \cB)) \to T(\cA, \cB), \qquad (F_1, F_2) \mapsto (F_1(\bM), F_2(\bN))
\]
and
\[
T(\cA, \cB) \to {\rm LEx}^{\otimes}( (\Mod_\ub^\rf, \Mod_\uspin^\rf), (\cA, \cB)), \qquad (A,B) \mapsto ((M,N) \mapsto (S_M(A), S_N(B)))
\]
are mutually quasi-inverse equivalences.
\end{theorem}

\begin{proof}
Same proof as \cite[Theorem 3.4.2]{infrank}.
\end{proof}

\subsection{Twisted Lie algebras}

Let $\fa=\bC^{\infty} \oplus \Sym^2(\bC^{\infty})$, with the usual $\GL_{\infty}$ action. Define a bracket on $\fa$ by $[(v, f), (v', f')] = (0, vv')$, where $v,v' \in \bC^{\infty}$ and $f,f' \in \Sym^2(\bC^{\infty})$. This bracket is ``graded-anticommutative'' in that it is anti-commutative in even degrees and commutative in odd degrees. Let $\cU(\fa)$ be its universal enveloping algebra, defined in the obvious manner. We only consider $\fa$- or $\cU(\fa)$-modules with compatible polynomial $\GL_{\infty}$ action.

\begin{theorem}
We have an equivalence of abelian categories $\Mod_{\cU(\fa)} \simeq \Mod_\uspin$. Under this equivalence, the simple $\cU(\fa)$-module $\bS_{\lambda}$ corresponds to $\Delta_{\lambda}$ if we make the identification $\Mod^\rf_\uspin = \Rep^\spin(\fg)$ from Theorem~\ref{thm:kernel-spin}.
\end{theorem}

\begin{proof}
Let $\FB$ be the groupoid of finite sets, and let $\Mod_{\FB}$ be the category of functors $\FB \to \Vec$. Schur--Weyl duality induces a symmetric monoidal equivalence between $\Mod_{\FB}$ and the category of polynomial $\GL_{\infty}$-representations (see \cite[(5.4.5)]{expos}). Let $\wt{\sU} \in \Mod_{\FB}$ be defined as follows: $\wt{\sU}_S$ has a basis consisting of elements $e_{U,\Gamma}$ where $U$ is a subset of $S$ equipped with a total order and $\Gamma$ is a perfect matching on $S \setminus U$. Let $\sU$ be defined as follows: $\sU_S$ is the quotient of $\wt{\sU}_S$ by the relations $e_{U,\Gamma}+e_{U',\Gamma}=e_{U'',\Gamma'}$, where $U'$ is obtained from $U$ by switching (in the order) two consecutive elements $i$ and $j$, $U''$ is $U \setminus \{i,j\}$, and $\Gamma'$ is $\Gamma$ together with the additional edge $(i,j)$.

We claim that $\cU(\fa)$ corresponds to $\sU$ under Schur--Weyl duality.
Under this equivalence, the space $\fa = \bC^\infty \oplus \Sym^2(\bC^\infty)$ corresponds to the functor $M \colon \FB \to \Vec$ given by
\begin{displaymath}
M_S = \begin{cases} \bC & \text{if $S$ has cardinality 1 or 2} \\ 0 & \text{otherwise} \end{cases}
\end{displaymath}
where $\bC$ is the trivial representation. Now, if $V,W \in \Mod_{\FB}$ then, by definition, we have
\begin{displaymath}
(V \otimes W)_S = \bigoplus_{S = A \amalg B} V_A \otimes W_B.
\end{displaymath}
It follows that $(\rT M)_S$ (where $\rT$ is the tensor algebra) has for a basis the set of partitions of $S$ into sets of size~1 and~2, together with a total order on the pieces of the partition. Now, $\cU(\fa)$ is obtained from $\rT(\fa)$ by killing $\lw^2(\fa_2)$, the anti-symmetric part of $(\fa_1 \otimes \fa_2) \oplus (\fa_2 \otimes \fa_1)$, and identifying $\Sym^2(\fa_1)$ with $\fa_2$ in the natural way. Under Schur--Weyl, the first two relations correspond to ignoring the ordering on 2-element sets; thus the quotient of $\rT M$ by these relations yields $\wt{\sU}$. Further quotienting by the Schur--Weyl dual of the third relations gives $\sU$. This proves the claim that $\cU(\fa)$ corresponds to $\sU$.

We now claim that $\Mod_{\sU}$ is equivalent to $\Mod_\uspin$. Let $\wt{\uspin}$ be the category whose objects are finite sets and whose morphisms are spin-Brauer diagrams. Then $\Mod_{\wt{\sU}}$ is equivalent to $\Mod_{\wt{\uspin}}$; this can be proved following the reasoning in \cite[\S 2.4]{infrank}. From this, the claim easily follows: imposing the Clifford relations on $\wt{\uspin}$ amounts to passing from $\wt{\sU}$ to $\sU$.
\end{proof}

\begin{corollary} \label{cor:spin-ext}
Let $c^\lambda_{\mu,\nu}$ be the Littlewood--Richardson coefficient. We have
\[
\dim \ext^i(\Delta_\mu, \Delta_\lambda) = \sum_{\substack{\nu = \nu^\dagger,\\ 2i = |\nu| + \rank(\nu)}} c^\lambda_{\mu, \nu}.
\]
\end{corollary}

\begin{proof}
If $\bK_\bullet$ is the minimal free resolution of $\bC$ over $\cU(\fa)$, then 
\[
\bK_i = \cU(\fa) \otimes \bigoplus_{\substack{\nu = \nu^\dagger,\\ 2i = |\nu| + \rank(\nu)}} \bS_\nu
\]
(see Proposition~\ref{prop:spin-koszul} below). Then $\bK_\bullet \otimes \bS_\lambda$ gives a minimal free resolution of the simple $\cU(\fa)$-module $\bS_\lambda$. Applying $\hom_{\cU(\fa)}(\bS_\mu, -)$ kills the differentials in this complex, and we get the sum written above.
\end{proof}

\begin{remark}
Let $\fa'$ be the super vector space $\bC^{\infty|\infty}[1] \oplus \Sym^2(\bC^{\infty|\infty})$. This has a natural bracket, defined similarly to the one on $\fa$, that gives $\fa'$ the structure of a Lie superalgebra. The category of $\fa$-modules is equivalent to a certain category of $\fa'$-modules, and in this way one can connect $\Rep^{\spin}(\fg)$ to representations of $\fa'$.
\end{remark}

\subsection{Half-spinor representations}

Put
\begin{displaymath}
\ol{\fg} = \lw^2(\ol{\bV}) = \lw^2(\bW) \oplus (\bW \otimes \bW_*) \oplus \lw^2(\bW_*),
\end{displaymath}
Then $\ol{\fg} \subset \fg$ is a Lie subalgebra, $\ol{\bV} \subset \bV$ is stable by $\ol{\fg}$, and thus a representation, called the {\bf standard representation}. A representation of $\ol{\fg}$ is {\bf algebraic} if it is a subquotient of a finite direct sum of tensor powers of the standard representation. We write $\Rep(\ol{\fg})$ for the category of such representations. 

The restriction of the action to $\ol{\fg}$ preserves the natural $\bZ/2$-grading on $\Delta$, and so $\Delta$ splits into a sum of two subrepresentations $\Delta^+$ (the even piece) and $\Delta^-$ (the odd piece), the {\bf half-spinor representations} of $\ol{\fg}$. The action of the $h$'s preserves the $\bZ$-grading on $\Delta$. By restriction, we get maps of $\ol{\fg}$-modules 
\[
\ol{\bV} \otimes \Delta^+ \to \Delta^-, \qquad \ol{\bV} \otimes \Delta^- \to \Delta^+,
\]
both of which are surjective.

A representation of $\ol{\fg}$ is {\bf spin-algebraic} if it appears as a subquotient of a finite direct sum of representations of the form 
\[
\ol{T}^n = \ol{\bV}^{\otimes n} \otimes \Delta.
\]
The space $\ol{T}^n$ breaks up into a sum $\ol{T}^{n,+} \oplus \ol{T}^{n,-}$, where $\ol{T}^{n,\pm} = \ol{\bV}^{\otimes n} \otimes \Delta^{\pm}$. Any spin-algebraic representation appears as a subquotient of a finite direct sum of the representations $\ol{T}^{n,+}$ (and similarly for $\ol{T}^{n,-}$).  We write $\Rep^{\spin}(\ol{\fg})$ for the category of spin-algebraic representations of $\ol{\fg}$. It is an abelian category, and is naturally a module over the tensor category $\Rep(\ol{\fg})$, i.e., we have a bifunctor given by tensor product
\begin{align*}
\otimes \colon \Rep(\ol{\fg}) \times \Rep^\spin(\ol{\fg}) \to \Rep^\spin(\ol{\fg}).
\end{align*}

\begin{proposition}
Any non-zero $\ol{\fg}$-submodule of $\ol{T}^{n,+}$ has magnitude $n$ and any non-zero $\ol{\fg}$-submodule of $\ol{T}^{n,-}$ has magnitude $n+1$.
\end{proposition}

\begin{proof}
The proof is similar to the proof of Proposition~\ref{prop:spin-mag}, but with the following statements: any non-zero submodule of $\ol{\bV}^{\otimes n} \otimes \Delta^+$ intersects $\ol{\bV}^{\otimes n} \otimes \Delta^0$, and any non-zero submodule of $\ol{\bV}^{\otimes n} \otimes \Delta^-$ intersects $\ol{\bV}^{\otimes n} \otimes \Delta^1$.
\end{proof}

We define $\ol{T}^{[n]}$ to be the intersection of the kernels of the maps $\ol{T}^n \to \ol{T}^{n-1}$. This breaks up as $\ol{T}^{[n],+} \oplus \ol{T}^{[n],-}$, where $\ol{T}^{[n], \pm}$ is the intersection of the kernels of maps $\ol{T}^{n,\pm} \to \ol{T}^{n-1,\mp}$. Define
\begin{displaymath}
\ol{\Delta}_{\lambda} = \Hom_{S_n}(\bM_{\lambda}, \ol{T}^{[n]}), \qquad
\ol{\Delta}_{\lambda}^{\pm} = \Hom_{S_n}(\bM_{\lambda}, \ol{T}^{[n],\pm})
\end{displaymath}
Note that $\ol{\Delta}_{\lambda} = \ol{\Delta}^+_{\lambda} \oplus \ol{\Delta}_{\lambda}^-$. 

To analyze $\Delta^{\pm}_\lambda$, we first need to understand the finite-dimensional case. So let $V$ be a $2n$-dimensional orthogonal space, and define $\Delta^{(n), \pm}_\lambda$ in an analogous way.

\begin{lemma} \label{lem:even-spinor-ker}
Let $\delta_\pm$ be the highest weight for $\Delta^{(n),\pm}$. If $\ell(\lambda) \le n$ then $\Delta^{(n),\pm}_\lambda$ is an irreducible representation of $\Spin(V)$ of highest weight $\lambda + \delta_\pm$. 
\end{lemma}

\begin{proof}
Same as the proof of Lemma~\ref{lem:odd-spinor-ker} except that we use \cite[Proposition 4.2]{exceptional}.
\end{proof}

\begin{proposition}
The representations $\ol{\Delta}_\lambda^{\pm}$ are a complete irredundant set of simple objects of $\Rep^\spin(\ol{\fg})$. 
\end{proposition}

\begin{proof}
The proof is similar to the proof of Proposition~\ref{prop:spin-simples}.
\end{proof}

Let $\{\pm\}$ be the category with objects labeled $+$ and $-$ and no non-identity morphisms. Define $\Mod_{\dspin \times \{\pm\}}^\rf$ in the same way as $\Mod_\dspin^\rf$ and define a representation $\ol{\cK}$ of $\dspin \times \{\pm\}$ by 
\begin{align*}
\ol{\cK}_{(L,+)} = \begin{cases} \ol{\bV}^{\otimes L} \otimes \Delta^+ & \text{if $\# L$ is even}\\
\ol{\bV}^{\otimes L} \otimes \Delta^- & \text{if $\# L$ is odd}
\end{cases}, \qquad
\ol{\cK}_{(L,-)} = \begin{cases} \ol{\bV}^{\otimes L} \otimes \Delta^- & \text{if $\# L$ is even}\\
\ol{\bV}^{\otimes L} \otimes \Delta^+ & \text{if $\# L$ is odd}
\end{cases}.
\end{align*}
On morphisms of $\dspin \times \{\pm\}$, $\ol{\cK}$ is defined using the maps $t_i$ as in the definition of $\cK$. Then:

\begin{theorem} \label{thm:kernel-spin2}
The kernel $\ol{\cK}$ induces mutually quasi-inverse anti-equivalences between $\Mod_{\dspin \times \{\pm\}}^\rf$ and $\Rep^{\spin}(\ol{\fg})$, as tensor categories.
\end{theorem}

\begin{proof}
Similar to the proof of Theorem~\ref{thm:kernel-spin}.
\end{proof}

Consider a pair of categories $(\cA, \cB)$ as in \S \ref{ss:spinuniv}. Define $\ol{T}(\cA, \cB)$ to be the category whose objects are tuples $(A, \omega, B_+, B_-, t)$ where $(A, \omega, B_+ \oplus B_-, t) \in T(\cA, \cB)$ and the morphism $t$ decomposes as $A \otimes B_+ \to B_-$ and $A \otimes B_- \to B_+$. In other words, $B_+ \oplus B_-$ is a $\bZ/2$-graded module over $\Cl(A)$. The morphisms in this category are defined in a way similar to those in $T(\cA, \cB)$.

\begin{theorem} 
The data of a left-exact tensor functor $(\Rep(\ol{\fg}), \Rep^\spin(\ol{\fg}))$ to $(\cA, \cB)$ is the same as an object of $\ol{T}(\cA, \cB)$. 
\end{theorem}

\begin{proof}
Similar to the proof of \cite[Theorem 3.4.2]{infrank}.
\end{proof}

\section{Infinite rank oscillator representations}

\subsection{Transpose duality} \label{ss:transpose}

All of the objects constructed in \S\ref{sec:spinor} can be thought of as representations of $\GL(\bW)$, and so belong to the category $\Rep(\GL)$ described in \cite[\S 3.1]{infrank}. We can use the transpose duality functor \cite[(3.3.8)]{infrank} to define an auto-equivalence on $\Rep(\GL)$ which sends the simple object $V_{\lambda, \mu}$ to $V_{\lambda^\dagger, \mu^\dagger}$ and which is an anti-symmetric monoidal equivalence in each variable.

In particular, we can apply this equivalence to the objects constructed in \S\ref{sec:spinor} to get a completely parallel theory of oscillator representations. In the rest of the section, we state the relevant results and omit the proofs since they follow from duality.

\subsection{Basic definitions}

Let $\bW=\bC^{\infty}$ and let $\bW_*$ be its restricted dual and put 
\[
\ol{\bV}=\bW \oplus \bW_*, \qquad 
\bV = \bC \uplus \ol{\bV}.
\]
Thus $\bV$ is a $\bZ/2$-graded vector space; in fact, we regard it as $\bZ$-graded with $\deg(\bC) = 0$ and $\deg(\bW) = 1$ and $\deg(\bW_\ast)=-1$. As before, we let $\be$ be a basis vector of $\bC$. We put a form $\omega$ on $\ol{\bV}$ by 
\[
\omega((v, f), (v', f')) = f'(v) - f(v').
\]
We extend this to a supersymmetric form, also called $\omega$, to $\bV$ by $\omega(\be,\be) = 1$ and $\omega(\be,v) = 0$ for all $v \in \ol{\bV}$. Put
\begin{align*}
\fg = \fosp(1|2\infty) &\cong \Sym^2(\bV)\\
& = \Sym^2(\bW) \oplus \bW \oplus (\bW \otimes \bW_*) \oplus \bW_* \oplus \Sym^2(\bW_*),
\end{align*}
We name some elements of $\fg$:
\begin{itemize}
\item For $v, w \in \bW$ we let $x_{v,w}=vw \in \Sym^2(\bW) \subset \fg$ and $x_v = v \otimes \be \in \bW \otimes \bC \subset \fg$; 
\item For $v \in \bW$ and $\phi \in \bW_*$ we let $h_{v,\phi}=v \otimes \phi \in \bW \otimes \bW_* \subset \fg$; 
\item For $\phi, \psi \in \bW_*$ we let $y_{\phi, \psi}=\phi\psi \in \Sym^2(\bW_*) \subset \fg$ and $y_\phi = \be \otimes \phi \in \bC \otimes \bW \subset \fg$. 
\end{itemize}

Define a map $\fg \to \fgl(\bV)$ as follows. Suppose $u \in \bW \subset \ol{\bV}$. Then
\begin{displaymath}
x_{v,w} u=0, \qquad x_v u = 0, \qquad h_{v,\phi} u = \phi(u) v, \qquad y_\phi u = \phi(u), \qquad y_{\phi,\psi} u = \psi(u) \phi + \phi(u) \psi.
\end{displaymath}
We define the action of $\eta \in \bW_*$ in an analogous manner:
\begin{displaymath}
x_{v,w} \eta= \eta(w)v + \eta(v)w, \qquad x_v \eta = \eta(v), \qquad h_{v,\phi} \eta = \eta(v) \phi, \qquad y_\phi \eta = 0, \qquad y_{\phi,\psi} \eta = 0.
\end{displaymath}
We also set
\begin{displaymath}
x_{v,w} \be=0, \qquad x_v \be = v, \qquad h_{v,\phi} \be = 0, \qquad y_\phi \be = \phi, \qquad y_{\phi,\psi} \be=0.
\end{displaymath}
Then $\fg$ is closed under the Lie superbracket on $\fgl(\bV)$ and so is a Lie superalgebra. It preserves the form $\omega$ on $\bV$ which is why we have called it $\osp(1|2\infty)$. We call $\bV$ the {\bf standard representation} of $\fg$. We say that a representation of $\fg$ is {\bf algebraic} if it appears as a subquotient of a finite direct sum of tensor powers of the standard representation.

Let $\nabla^n=\Sym^n{\bW}$, and let $\nabla=\bigoplus_{n \ge 0} \nabla^n$ be the symmetric algebra on $\bW$. For $v \in \bW$, we let $X_v$ be the operator on $\nabla$ given by 
\[
X_v(w)=vw.
\]
The operators $X_v$ and $X_w$ commute. For $\phi \in \bW_*$, we let $Y_{\phi}$ be the operator on $\nabla$ given by
\begin{displaymath}
Y_{\phi}(v_1 \cdots v_n)=\sum_{i=1}^n \phi(v_i) v_1 \cdots \wh{v}_i \cdots v_n.
\end{displaymath}
The operators $Y_{\phi}$ and $Y_{\psi}$ also commute and $X_vY_{\phi} - Y_{\phi}X_v = -\phi(v)$. We let 
\[
H_{v,\phi}=X_v Y_{\phi};
\]
this is the usual action of the element $v \phi \in \mf{gl}(\bW)$ on $\nabla$. Finally, define $D$ by 
\[
D(v_1 \cdots v_n) = (-1)^n v_1 \cdots v_n.
\]
Then $D$ supercommutes with all $X_v$ and all $Y_\phi$.

We now define a representation $\rho$ of $\fg$ on $\nabla$, as follows:
\begin{align*}
&\rho(x_{v,w}) = X_v X_w, \qquad
\rho(x_v) = X_v D, \qquad
\rho(h_{v,\phi})=H_{v,\phi}-\tfrac{1}{2} \phi(v),\\
&\rho(y_\phi) = D Y_\phi, \qquad
\rho(y_{\phi,\psi})=Y_{\phi} Y_{\psi}.
\end{align*}
We leave it to the reader to verify that this is a well-defined representation. This is the {\bf oscillator representation} of $\fg$. 

The map $\bV \to \End(\nabla)$ defined by $v \mapsto X_v$ and $\phi \mapsto Y_{\phi}$ and $\be \mapsto D$ is a map of $\fg$-representations. It follows that the map 
\[
\bV \otimes \nabla \to \nabla
\]
given by $v \otimes x \mapsto X_v x$ and $\phi \otimes x \mapsto Y_{\phi} x$ and $1 \otimes x \mapsto Dx$ is a map of $\fg$-representations.

We say that a representation of $\fg$ is {\bf osc-algebraic} if it appears as a subquotient of a finite direct sum of representations of the form 
\[
T^n = \bV^{\otimes n} \otimes \nabla.
\]
We write $\Rep^{\osc}(\fg)$ for the category of osc-algebraic representations of $\fg$. It is an abelian category, and is naturally a module over the tensor category $\Rep(\fg)$, i.e., we have a bifunctor given by tensor product
\begin{align*}
\otimes \colon \Rep(\fg) \times \Rep^\osc(\fg) \to \Rep^\osc(\fg).
\end{align*}

\subsection{Weyl's construction (infinite case)}

Let $t_i \colon T^n \to T^{n-1}$ be the map given by applying the map $\bV \otimes \nabla \to \nabla$ to the $i$th factor. An important property of these maps is that $t_i t_j - t_j t_i$ is induced from the map $\bV^{\otimes n} \to \bV^{\otimes (n-2)}$ given by the symplectic pairing on the $i$th and $j$th factors. Define $T^{[n]}$ to be the intersection of the kernels of the $t_i$. This is stable under the action of $S_n \times \fg$. Finally, define
\begin{displaymath}
\nabla_{\lambda}=\Hom_{S_n}(\bM_{\lambda}, T^{[n]})
\end{displaymath}
where $\bM_\lambda$ is the irreducible representation of $S_n$ indexed by $\lambda$. We now have a number of results that are analogous to those in \S\ref{sec:weyl-spin}.

\begin{proposition}
As $\lambda$ ranges over all partitions, the representations $\nabla_{\lambda}$ are a complete irredundant set of simple objects of $\Rep^{\osc}(\fg)$.
\end{proposition}

\begin{corollary}
Every object of $\Rep^\osc(\fg)$ has finite length.
\end{corollary}

\begin{proposition}
Any non-zero $\fg$-submodule of $T^n$ has magnitude $n$.
\end{proposition}

\begin{proposition}
We have the following:
\begin{displaymath}
\Hom_{\fg}(\nabla_\lambda, T^n) = \begin{cases}
\bM_{\lambda} & \textrm{if $n=\vert \lambda \vert$} \\
0 & \textrm{otherwise}
\end{cases}.
\end{displaymath}
\end{proposition}

\subsection{Diagram category}

An {\bf osc-Brauer diagram} between a set $L$ and $L'$ is a triple $(U, \Gamma, f)$ where 
\begin{itemize}
\item $U$ is a subset of $L$ equipped with a total order on $U$, 
\item $(\Gamma, f)$ is a directed Brauer diagram on $(L\setminus U) \amalg L'$, i.e., $\Gamma$ is a directed matching on $L \setminus U$ and $f \colon L \setminus (U \cup V(\Gamma)) \to L'$ is a bijection. Here $V(\Gamma)$ is the vertex set of $\Gamma$.
\end{itemize}
We just write $\Gamma$ for such diagrams, and think of the vertices in $U$ as circled. The main difference between osc-Brauer diagrams and spin Brauer diagrams is that the matching on $L$ is directed. Pictorially, we may represent this as follows
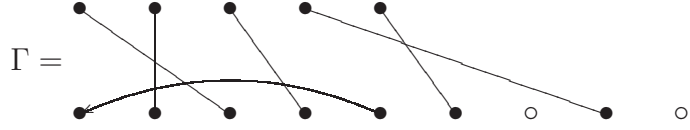
\[
\Gamma = \begin{xy}
(-28, 7)*{}="A1"; (-18, 7)*{}="B1"; (-8, 7)*{}="C1"; (2, 7)*{}="D1"; (12, 7)*{}="E1"; 
(-28, -7)*{}="A2"; (-18, -7)*{}="B2"; (-8, -7)*{}="C2"; (2, -7)*{}="D2"; (12, -7)*{}="E2"; (22, -7)*{}="F2"; (32, -7)*{}="G2"; (42, -7)*{}="H2"; (52, -7)*{}="I2";
"A1"*{\bullet}; "B1"*{\bullet}; "C1"*{\bullet}; "D1"*{\bullet}; "E1"*{\bullet}; 
"A2"*{\bullet}; "B2"*{\bullet}; "C2"*{\bullet}; "D2"*{\bullet}; "E2"*{\bullet}; "F2"*{\bullet}; "G2"*{\circ}; "H2"*{\bullet}; "I2"*{\circ};
"A1"; "C2"; **\dir{-};
"B1"; "B2"; **\dir{-};
"C1"; "D2"; **\dir{-};
"E1"; "F2"; **\dir{-};
"A2"; "E2"; **\crv{(-8,2)} ?>*\dir{>};
"H2"; "D1"; **\dir{-};
\end{xy}
\]
where we have used an arrowhead to denote a direction for the edge in the matching.

Suppose $\Gamma$ is an osc-Brauer diagram and $\Gamma'$ is obtained by switching the two consecutive elements $i<j$ of $U$ in the total order. Let $\Gamma''$ be the spin-Brauer diagram obtained by removing $i$ and $j$ from $U$ in $\Gamma$, and placing a directed edge from $i$ to $j$. The {\bf Weyl relation} is $\Gamma-\Gamma'=\Gamma''$.

The {\bf downwards oscillator category}, denoted $\dosc$, is the following $\bC$-linear category. The objects are finite sets. The set $\Hom_\dosc(L,L')$ is the quotient of the vector space spanned by the osc-Brauer diagrams by the Weyl relations and by the relation $\Gamma = (-1)^n \Gamma'$ if $\Gamma'$ is a directed Brauer diagram obtained from $\Gamma$ by reversing the orientation on $n$ different edges in the matching. Given $(U, \Gamma, f) \in \Hom_\dspin(L,L')$ and $(U', \Gamma', f') \in \Hom_\dspin(L', L'')$, define their composition to be $(U \cup f^{-1}(U'), \Gamma \cup f^{-1}(\Gamma'), g)$ where $g$ is the restriction of $f'f$ to $L \setminus (U \cup f^{-1}(U') \cup V(\Gamma) \cup V(f^{-1}(\Gamma')))$. The {\bf upwards oscillator category} is $\uosc = \dosc^{\rm op}$. A {\bf representation} of $\dosc$ is a linear functor $F \colon \dosc \to \Vec$. We let $\Mod_\dosc^\rf$ denote the abelian category of finite length representations of $\dosc$. 

Define a representation $\cK$ of $\dosc$ by $\cK_L=\bV^{\otimes L} \otimes \nabla$. An osc-Brauer diagram $\Gamma$ between $L$ and $L'$ induces a map $\cK_L \to \cK_{L'}$ by using the usual recipe on the Brauer part, and using the maps $\bV^{\otimes n} \otimes \nabla \to \bV^{\otimes (n-1)} \otimes \nabla$ on the $U$ part, with the ordering on $U$ specifying the order of these maps.

In \cite[(4.2.11)]{infrank}, we defined $\dsb$, which in our current notation is the subcategory of $\dosc$ where we only use osc-Brauer diagrams with $U = \emptyset$. This has a symmetric monoidal structure using disjoint union $\amalg$. Clearly, $\amalg$ extends to a bifunctor $\dsb \times \dosc \to \dosc$. Using \cite[(2.1.14)]{infrank}, we get convolution tensor products
\begin{equation} \label{eqn:amalg-spin}
\amalg_* \colon \Mod_\dsb^\rf \times \Mod_\dosc^\rf \to \Mod_\dosc^\rf,
\end{equation}

\begin{theorem} \label{thm:kernel-osc}
The kernel object $\cK$ induces mutually quasi-inverse anti-equivalences between $\Mod_\dosc^\rf$ and $\Rep^{\osc}(\fg)$. Furthermore, in both cases, the tensor product in \eqref{eqn:amalg-spin} is taken to the natural tensor product
\begin{displaymath}
\otimes \colon \Rep(\fg) \times \Rep^\osc(\fg) \to \Rep^\osc(\fg),
\end{displaymath}
\end{theorem}

\begin{proposition}
For every partition $\lambda$, the representation $\bS_\lambda(\bV) \otimes \nabla$ is injective in $\Rep^\osc(\fg)$, and is the injective envelope of the simple $\nabla_\lambda$. 
\end{proposition}

\begin{corollary}
For every $n \ge 0$, $T^n$ is an injective object of $\Rep^\osc(\fg)$.
\end{corollary}

\subsection{Universal property} \label{ss:oscuniv}

Consider categories $\cA$ and $\cB$ as in \S \ref{ss:spinuniv}. Suppose that $A \in \cA$ and $\omega \colon \Sym^2(\cA) \to \bC$ is an alternating bilinear form on $A$. If $\cA$ has infinite direct sums, we can form the Weyl algebra $\cW(A)$ as the usual quotient of the tensor algebra on $A$. We can then speak of $\cW(A)$-modules in $\cB$. Even if $\cA$ does not have infinite direct sums, we can still define the notion of a $\cW(A)$-module in $\cB$: it is an object $B$ of $\cB$ equipped with a morphism $t \colon A \otimes B \to B$ such that the two maps
\begin{displaymath}
f, g \colon A \otimes A \otimes B \to B 
\end{displaymath}
given by
\begin{displaymath}
f(x \otimes y \otimes m)=t(x \otimes t(y \otimes m))-t(y \otimes t(x \otimes m)), \qquad
g(x \otimes y \otimes m) = \omega(x, y) m
\end{displaymath}
agree. We let $T'(\cA, \cB)$ be the category whose objects are tuples $(A, \omega, B, t)$ as above. We write $(A, B)$ for an object of $T'(\cA, \cB)$ when there is no danger of confusion.

Given $(A, B) \in T'(\cA, \cB)$, define $\cK(A) \colon \dsb \to \cA$ by $L \mapsto A^{\otimes L}$ and similarly, define $\cK(B) \colon \dosc \to \cB$ by $L \mapsto A^{\otimes L} \otimes B$. For an object $M$ of $\Mod_\usb^\rf$ and an object $N$ of $\Mod_\uosc^\rf$, define
\begin{align*}
S_M(A) = M \otimes^\ub \cK(A),\qquad S_N(B) = N \otimes^\ub \cK(B).
\end{align*}
Then $(M,N) \mapsto (S_M(A), S_N(B))$ defines a left-exact tensor functor $(\Mod_\usb^\rf, \Mod_\uosc^\rf) \to (\cA, \cB)$.

\begin{theorem} \label{thm:osc-univ}
To give a left-exact tensor functor $(\Rep(\fg), \Rep^\osc(\fg))$ to $(\cA, \cB)$ is the same as giving an object of $T'(\cA, \cB)$. More precisely, letting $\bM$ be the object of $\Mod_\usb^\rf$ corresponding to $\bV$ in $\Rep(\fg)$ and letting $\bN$ be the object of $\Mod_\uosc^\rf$ corresponding to $\nabla$ in $\Rep^\osc(\fg)$, the functors
\[
{\rm LEx}^{\otimes}( (\Mod_\usb^\rf, \Mod_\uosc^\rf), (\cA, \cB)) \to T'(\cA, \cB), \qquad (F_1, F_2) \mapsto (F_1(\bM), F_2(\bN))
\]
and
\[
T'(\cA, \cB) \to {\rm LEx}^{\otimes}( (\Mod_\usb^\rf, \Mod_\uosc^\rf), (\cA, \cB)), \qquad (A,B) \mapsto ((M,N) \mapsto (S_M(A), S_N(B)))
\]
are mutually quasi-inverse equivalences.
\end{theorem}

\subsection{Twisted Lie algebras}

Let $\fa$ be the twisted Lie algebra $\bC^\infty \oplus \lw^2(\bC^\infty)$ with the Lie bracket $[(v, f), (v', f')] = (0, v \wedge v')$. This is $\bZ$-graded with $\deg(\bC^\infty) = 1$ and $\deg(\lw^2(\bC^\infty))=2$. We only consider $\fa$-modules with a compatible polynomial $\GL_\infty$-action.

\begin{theorem}
We have an equivalence of abelian categories $\Mod_{\cU(\fa)} \simeq \Mod_\uosc$. If we make the identification $\Mod_\uosc^\rf = \Rep^\osc(\fg)$ from Theorem~\ref{thm:kernel-osc}, then the simple $\cU(\fa)$-module $\bS_\lambda$ is sent to $\nabla_\lambda$.
\end{theorem}

\begin{corollary}
We have
\[
\dim \ext^i_{\Rep^\osc(\fg)}(\nabla_\mu, \nabla_\lambda) = \sum_{\substack{\nu = \nu^\dagger,\\ 2i = |\nu| + \rank(\nu)}} c^\lambda_{\mu, \nu}.
\]
\end{corollary}

\subsection{Oscillator-spin duality} \label{sec:oscspin-dual}

The following result is an extension of orthogonal-symplectic duality \cite[(4.3.4)]{infrank}.

\begin{theorem}
We have an asymmetric monoidal equivalence of pairs of categories 
\begin{displaymath}
(\Rep(\fso(2\infty+1)), \Rep^\spin(\fso(2\infty+1))) \simeq (\Rep(\fosp(1|2\infty)), \Rep^\osc(\fosp(1|2\infty))).
\end{displaymath}
\end{theorem}

\begin{proof}
Given $(\cA, \cB)$, let $\cA^\dagger$ be the symmetric monoidal abelian category obtained from $\cA$ by swapping the symmetry $\tau$ on $\cA$ for $-\tau$. Then there is an obvious equivalence $T(\cA, \cB) \simeq T'(\cA^\dagger, \cB)$. Now we get the first equivalence, by using the universal properties Theorem~\ref{thm:spin-univ} and Theorem~\ref{thm:osc-univ}.
\end{proof}

\subsection{Half-oscillator representations} \label{sec:half-osc}

Put
\begin{align*}
\ol{\fg} = \fsp(2\infty) &\cong \Sym^2(\ol{\bV})\\
& = \Sym^2(\bW) \oplus (\bW \otimes \bW_*) \oplus \Sym^2(\bW_*),
\end{align*}
Then $\fg \subset \ol{\fg}$ is a Lie subalgebra, and $\ol{\bV} \subset \bV$ is stable by $\ol{\fg}$, and thus a representation, called the {\bf standard representation}. A representation of $\ol{\fg}$ is {\bf algebraic} if it is a subquotient of a finite direct sum of tensor powers of the standard representation. We write $\Rep(\ol{\fg})$ for the category of such representations. 

The action of $\fg$ on $\nabla$ preserves the natural $\bZ/2$-grading, and so $\nabla$ splits into a sum of two subrepresentations $\nabla^+$ and $\nabla^-$, the {\bf half-oscillator representations} of $\fg$. The action of the $h$'s preserves the $\bZ$-grading on $\nabla$. We have maps of $\fg$-modules 
\[
\ol{\bV} \otimes \nabla^+ \to \nabla^-, \qquad \ol{\bV} \otimes \nabla^- \to \nabla^+,
\]
both of which are surjective.

A representation of $\fg$ is {\bf osc-algebraic} if it appears as a subquotient of a finite direct sum of representations of the form 
\[
\ol{T}^n = \ol{\bV}^{\otimes n} \otimes \nabla.
\]
The space $\ol{T}^n$ breaks up as a sum $\ol{T}^{n,+} \oplus \ol{T}^{n,-}$, and any osc-algebraic representation appears as a subquotient of a finite direct sum of the representations $\ol{T}^{n,+}$ (and similarly for $\ol{T}^{n,-}$). We write $\Rep^{\osc}(\ol{\fg})$ for the category of osc-algebraic representations of $\ol{\fg}$. It is an abelian category, and is naturally a module over the tensor category $\Rep(\ol{\fg})$, i.e., we have a bifunctor given by tensor product
\begin{align*}
\otimes \colon \Rep(\ol{\fg}) \times \Rep^\osc(\ol{\fg}) \to \Rep^\osc(\ol{\fg}).
\end{align*}

We define $\ol{T}^{[n]}$ to be the intersection of the kernels of the maps $\ol{T}^n \to \ol{T}^{n-1}$. This breaks up as $\ol{T}^{[n],+} \oplus \ol{T}^{[n],-}$. Define
\begin{displaymath}
\ol{\nabla}_{\lambda}^{\pm} = \Hom_{S_n}(\bM_{\lambda}, \ol{T}^{[n], \pm}).
\end{displaymath}
The representations $\ol{\nabla}^{\pm}$ are a complete irredundant set of simple objects of $\Rep^{\osc}(\ol{\fg})$.

We define a representation $\cK$ of $\dosc \times \{\pm\}$ by 
\begin{align*}
\cK_{(L,+)} = \begin{cases} \ol{\bV}^{\otimes L} \otimes \nabla^+ & \text{if $\# L$ is even}\\
\ol{\bV}^{\otimes L} \otimes \nabla^- & \text{if $\# L$ is odd}
\end{cases}, \qquad
\cK_{(L,-)} = \begin{cases} \ol{\bV}^{\otimes L} \otimes \nabla^- & \text{if $\# L$ is even}\\
\ol{\bV}^{\otimes L} \otimes \nabla^+ & \text{if $\# L$ is odd}
\end{cases}.
\end{align*}

\begin{theorem}
The functors defined by $\cK$ give mutually quasi-inverse anti-equivalences between $\Mod_{\dosc \times \{\pm\}}^\rf$ and $\Rep^\osc(\fg)$.
\end{theorem}

Consider a pair of categories $(\cA, \cB)$ as in \S \ref{ss:spinuniv}. Define $\ol{T}'(\cA, \cB)$ to be the category whose objects are tuples $(A, \omega, B_+, B_-, t)$ where $(A, \omega, B_+ \oplus B_-, t) \in T'(\cA, \cB)$ and the morphism $t$ decomposes as $A \otimes B_+ \to B_-$ and $A \otimes B_- \to B_+$. In other words, $B_+ \oplus B_-$ is a $\bZ/2$-graded module over $\cW(A)$.

\begin{theorem} 
Giving a left-exact tensor functor $(\Rep(\ol{\fg}), \Rep^\spin(\ol{\fg})) \to (\cA, \cB)$ is the same as giving an object of $\ol{T}'(\cA, \cB)$. 
\end{theorem}

Finally, we have osc-spin duality:

\begin{theorem}
We have asymmetric monoidal equivalences of pairs of categories
\begin{displaymath}
(\Rep(\fso(2\infty)), \Rep^\spin(\fso(2\infty))) \simeq (\Rep(\fsp(2\infty)), \Rep^\osc(\fsp(2\infty))).
\end{displaymath}
\end{theorem}

\section{Invariant theory} \label{s:littlewood}

The goal of this section and the next is to extend the results of \cite{lwood} to the categories that we just defined.

\subsection{Spinors} \label{sec:howe-dual}

Let $E$ be a vector space of dimension $m$ and define a supersymmetric orthogonal form $\tilde{\omega}$ on the $\bZ/2$-graded vector space 
\[
\tilde{E} = \bC \uplus (E \oplus E^*)
\]
by
\[
\tilde{\omega}((c, e,f), (c, e', f')) = cc' + f'(e) - f(e').
\]
The Lie superalgebra $\osp(\tilde{E}) \cong \osp(1|2m)$ has a natural $\bZ$-grading supported on $[-2,2]$:
\begin{align*} 
\osp(\tilde{E}) = \Sym^2(E^*) \oplus E^* \oplus \fgl(E) \oplus E \oplus \Sym^2(E).
\end{align*}

Let $V$ be a vector space of dimension $N$ with a nondegenerate symmetric bilinear form. The tensor product of the two forms is a nondegenerate supersymmetric bilinear form on 
\[
V \otimes \tilde{E} = V \uplus (V \otimes (E \oplus E^*)),
\]
and we get an embedding 
\[
\fso(V) \times \osp(\tilde{E}) \subset \osp(V \otimes \tilde{E}) \cong \osp(N|2mN).
\] 
Let $W$ be a maximal isotropic subspace of $V$. Then $W \uplus (E \otimes V)$ is a maximal isotropic subspace of $V \otimes \tilde{E}$. As a vector space, the superspinor representation of $\osp(V \otimes \tilde{E})$ is 
\[
\fM = \Sym(E \otimes V) \otimes \bigwedge^\bullet(W).
\]
We can identify $\Delta = \bigwedge^\bullet(W)$ with the spinor representation of $\Pin(V)$, the double cover of the orthogonal group $\bO(V)$. Consider the action of $\osp(\tilde{E})$ on $\fM$. The action of $\Sym^2(E)$ is multiplication by the $\fso(V)$-invariant polynomials $\Sym^2(E) \subset \Sym^2(E \otimes V)$ and the action of $E$ comes from the natural inclusion $\Delta \subset V \otimes \Delta$:
\begin{align*}
E \otimes \Sym^i(E \otimes V) \otimes \Delta &\subset E \otimes V \otimes \Sym^i(E \otimes V) \otimes \Delta\\
&\to \Sym^{i+1}(E \otimes V) \otimes \Delta.
\end{align*}
There is a non-degenerate bilinear form on the spinor representation $\fM$ which is preserved by $\osp(V \otimes \tilde{E})$. The actions of $E^*$ and $\Sym^2(E^*)$ are adjoint to the actions of $E$ and $\Sym^2(E)$ under this bilinear form, and the action of $\fgl(E)$ is the usual one tensored with $N/2$ times the trace function (the even subalgebra of $\osp(\tilde{E})$ is $\fso(E \oplus E^*)$ and the action of this subalgebra ignores $\Delta$, and is described in \cite[\S 8.2]{goodman}). Let $\delta$ be the highest weight of $\Delta$. Given a partition $\lambda$, which we interpret as a dominant weight of $\Pin(V)$, let $V_{\lambda + \delta}$ be the irreducible highest weight representation of $\Pin(V)$ with highest weight $\lambda + \delta$, and set 
\[
\fM_\lambda = \hom_{\Pin(V)}(V_{\lambda + \delta}, \fM).
\] 
Note that $\fg^\dagger(E^*) = E^* \oplus \Sym^2(E^*)$ is a subalgebra of $\osp(\tilde{E}) \cong \osp(1|2m)$ and hence act on $\fM$. Define lowest-weight to mean annihilated by the lower-triangular part $\fg^\dagger(E^*)$. Let $\fM^-$ be the subspace of $\fM$ killed by $\fg^\dagger(E^*)$, i.e., the space of lowest-weight vectors. Then we have a multiplication map
\[
m \colon \rU(\fg^\dagger(E)) \otimes \fM^- \to \fM.
\]

\begin{proposition} 
\phantomsection \label{prop:mu-minus}
\begin{enumerate}[\rm (a)]
\item $m$ is surjective.
\item As a representation of $\fgl(E) \times \Pin(V)$, we have $\fM^- = \bigoplus_{2\ell(\lambda) \le \dim V} \bS_\lambda(E) \otimes V_{\lambda + \delta}$.
\item $\fM$ is a semisimple $(\osp(\tilde{E}) \times \Pin(V))$-module, and decomposes as
\[
\fM = \bigoplus_{\substack{\lambda\\ 2\ell(\lambda) \le \dim(V)}} \fM_\lambda \otimes V_{\lambda + \delta},
\]
where $\fM_\lambda$ is a simple lowest-weight representation and $\fM_\lambda \cong \fM_\mu$ if and only if $\lambda = \mu$.
\end{enumerate}
\end{proposition}

\begin{proof}
(a) Say an element of $\fM$ has degree $d$ if it lives in $\Sym^d(E \otimes V) \otimes \bigwedge^\bullet(W)$. The action of $\fg^\dagger(E^*)$ strictly decreases the degree of an element, so if an element $v$ is not in $\fM^-$, there exists $x \in \fg^\dagger(E^*)$, such that $xv \ne 0$ has smaller degree. By adjointness, $v$ can be generated by $xv$, and by induction on $d$, $xv$ can be generated by $\fM^-$. Now, using the decomposition $\rU(\osp(\tilde{E})) = \rU(\fg^\dagger(E^*)) \otimes \rU(\fgl(E)) \otimes \rU(\fg^\dagger(E))$ which comes from the Poincar\'e--Birkhoff--Witt theorem, we see that every element that can be generated by $\fM^-$ using $\rU(\osp(\tilde{E}))$ can be done so using just $\rU(\fg^\dagger(E))$.

(b) First, $\fg^\dagger(E) = E^* \oplus \Sym^2(E^*)$ and the action of $\Sym^2(E^*)$ is only on $\Sym(E \otimes V)$. The kernel of this action is the space of harmonic polynomials discussed in \cite[Theorem 9.1]{goodman}. Explicitly, it has a decomposition
\[
\bigoplus_{2\ell(\lambda) \le \dim V} \bS_\lambda(E) \otimes \bS_\lambda(V).
\]
To take the kernel by the action of $E^*$ on $\bS_\lambda(E) \otimes \bS_\lambda(V)$, we need to compute the kernel of 
\[
\bS_\lambda(E) \otimes \bS_\lambda(V) \otimes \Delta \to \bigoplus_\mu \bS_\mu(E) \otimes \bS_\mu(V) \otimes \Delta
\]
using the maps $t$ discussed in \S\ref{sec:weyl-spin}. From Lemmas~\ref{lem:odd-spinor-ker} and~\ref{lem:even-spinor-ker}, we deduce that the kernel is $\bS_\lambda(E) \otimes V_{\lambda + \delta}$.

(c) It follows immediately from (a) and (b) that $\fM$ has a decomposition as above where $\fM_\lambda$ is a lowest-weight representation generated by $\bS_\lambda(E)$. If $\fM_\lambda$ were reducible, then it would have lowest-weight vectors besides $\bS_\lambda(E)$; however, the decomposition of $\fM^-$ is multiplicity-free, so this does not happen.
\end{proof}

\begin{remark}
The decomposition in (c) is a special case of the one in \cite[Theorem A.1]{ckw} when $N$ is even.
\end{remark}

\subsection{Oscillators}

Now define an orthogonal form $\omega'$ on 
\[
E' = \bC \oplus E \oplus E^*
\]
by
\[
\omega'((c, e, f), (c', e', f')) = cc' + f'(e) + f(e').
\]
The Lie algebra $\fso(E') \cong \fso(1+2m)$ has a natural $\bZ$-grading supported on $[-2,2]$:
\begin{align*} 
\fso(E') = \lw^2(E^*) \oplus E^* \oplus \fgl(E) \oplus E \oplus \lw^2(E).
\end{align*}

If $N$ is even, let $U$ be a symplectic vector space of dimension $N$. If $N$ is odd, let $U$ be a superspace of dimension $(1|N-1)$ with a superorthogonal form (so the odd part is a symplectic space in the usual sense).

For uniformity of notation, we write $\osp(U)$ for the Lie (super)algebra preserving the form on $U$. We also let $\Mp(U)$ denote the metaplectic cover of the Lie supergroup $\bO\Sp(U)$. Rather than define these groups precisely, we just need to know that the half-oscillator representations of \S\ref{sec:half-osc} always appear together, that is, there is a nontrivial element in $\Mp(U)$ that takes one to the other (so this is only an issue when $N$ is even). The tensor product of $\omega'$ with this orthosymplectic form gives an orthosymplectic form on $U \otimes E'$, and we get an embedding 
\[
\osp(U) \times \fso(E') \subset \fosp(U \otimes E') \cong \begin{cases} \fosp(2m+1|(2m+1)(N-1)) & N \text{ odd} \\ 
\fsp(2mN+N) & N \text{ even} \end{cases}.
\]

Let $W$ be a maximal isotropic subspace of $U$ (now just thought of as an ungraded vector space). Then $W \oplus (E \otimes U)$ is a maximal isotropic subspace of $U \otimes E'$. We identify $\nabla = \Sym(W)$ with the oscillator representation of $\osp(U)$. As a vector space, the oscillator representation of $\fosp(U \otimes E')$ is 
\[
\fN = \Sym(E \otimes U) \otimes \Sym(W).
\]
This can be analyzed in the same way that $\fM$ was analyzed in \S\ref{sec:howe-dual}, so we omit the proofs and just state the analogous results.

Let $\eta$ be the highest weight of the oscillator representation $\nabla$. Given a partition $\lambda$, which we interpret as a dominant weight of $\Mp(U)$, set 
\[
\fN_\lambda = \hom_{\Mp(U)}(V_{\lambda + \eta}, \fN).
\]
Note that $\fg(E^*) = E^* \oplus \lw^2(E^*)$ is a subalgebra of $\fso(E')$ and hence acts on $\fN$. Let $\fN^-$ be the subspace of $\fN$ annihilated by $\fg(E^*)$. Then we have a multiplication map
\[
m \colon \rU(\fg(E)) \otimes \fN^- \to \fN.
\]
\begin{proposition} 
\phantomsection \label{prop:fN-minus}
\begin{enumerate}[\rm (a)]
\item $m$ is surjective.
\item As a representation of $\fgl(E) \times \Mp(U)$, we have $\fN^- = \bigoplus_{2\ell(\lambda) \le \dim V} \bS_\lambda(E) \otimes V_{\lambda + \eta}$.
\item $\fN$ is a semisimple $(\fso(E') \times \Mp(U))$-module, and decomposes as
\[
\fN = \bigoplus_{\substack{\lambda\\ 2\ell(\lambda) \le \dim(V)}} \fN_\lambda \otimes V_{\lambda + \eta},
\]
where $\fN_\lambda$ is a simple lowest-weight representation and $\fN_\lambda \cong \fN_\mu$ if and only if $\lambda = \mu$.
\end{enumerate}
\end{proposition}

\subsection{Koszul complexes}

Let $\fg(E) = E \oplus \lw^2(E)$ be the free $2$-step nilpotent Lie algebra on $E$. As a vector space, its universal enveloping algebra is (for the second equality, one can use \eqref{eqn:wedge-plethysm} together with the Pieri rule \cite[(3.10)]{expos})
\[
\rU(\fg(E)) = \Sym(E) \otimes \Sym(\lw^2(E)) = \bigoplus_{\lambda} \bS_\lambda E.
\]
The Chevalley--Eilenberg complex $\bK(\fg(E))_\bullet$ of $\fg(E)$ is given by $\bK(\fg(E))_i = \rU(\fg(E)) \otimes \bigwedge^i \fg(E)$ and is a projective resolution of the trivial module $\bC$. But this resolution is not minimal, i.e., the entries of the differentials do not belong to the augmentation ideal of $\rU(\fg(E))$. 

\begin{proposition} \label{prop:osc-koszul}
The terms of the minimal subcomplex $\ol{\bK}(\fg(E))_\bullet$ of $\bK(\fg(E))_\bullet$ are 
\[
\ol{\bK}(\fg(E))_i = \rU(\fg(E)) \otimes \bigoplus_{\substack{\lambda = \lambda^\dagger, \\ 2i = |\lambda| + \rank(\lambda)}} \bS_\lambda E.
\]
\end{proposition}

\begin{proof}
This follows from either \cite{reducedkoszul} or \cite{sigg}.
\end{proof}

We have a natural surjection of algebras $\rU(\fg(E)) \to \Sym(E)$.

\begin{lemma} \label{lem:torcalc}
$\Tor^{\rU(\fg(E))}_i(\Sym(E), \bC) = \bigwedge^i(\bigwedge^2(E))$.
\end{lemma}

\begin{proof}
Consider the symmetric algebra $S = \Sym(E \oplus \bigwedge^2( E))$. Let $I$ be the ideal generated by $\lw^2(E)$. Then $S/I = \Sym(E)$ as a vector space. Since $I$ is linear, it is  generated by a regular sequence, and so $\Tor^S_i(S/I, \bC) = \bigwedge^i(\lw^2(E))$. The Poincar\'e--Birkhoff--Witt theorem gives a $\GL(E)$-equivariant flat family with generic fiber $\rU(\fg(E))$ and special fiber $S$. This degeneration takes the quotient map $\rU(\fg(E)) \to \Sym(E)$ to $S \to S/I$. The dimension of Tor modules is upper semicontinuous for flat families, so we have
\[
\dim_\bC \Tor^{\rU(\fg(E))}_i(\Sym(E), \bC) \le \dim_\bC \Tor^S_i(S/I, \bC).
\]
Since $\Tor^S_\bullet(S/I,\bC)$ is multiplicity-free as a $\GL(E)$-representation, and the equivariant Euler characteristic is preserved in flat families, we get an isomorphism of $\GL(E)$-modules $\Tor^{\rU(\fg(E))}_i(\Sym(E), \bC) \cong \Tor^S_i(S/I, \bC)$. 
\end{proof}

Finally, let $\Delta$ be the spinor representation of $\fso(E')$. Then $\fg(E)$ is a nilpotent subalgebra of $\fso(E')$ and we can calculate the homology of the restriction of $\Delta$ using Kostant's theorem:

\begin{proposition} 
As a $\fgl(E)$-representation, we have 
\[
\Tor^{\rU(\fg(E))}_i(\Delta, \bC) = \bigwedge^i(\Sym^2(E)) \otimes \bC_{(-\frac{1}{2}, \dots, -\frac{1}{2})}.
\]
In particular, there is an acyclic $\GL(E)$-equivariant complex of free $\rU(\fg(E))$-modules with terms $\rU(\fg(E)) \otimes \bigwedge^i(\Sym^2(E))$.
\end{proposition}

\begin{proof}
We follow the exposition of Kostant's theorem from \cite[\S 2.1]{heisenberg}. In that notation, set $\lambda$ to be the highest weight of $\Delta$, which in coordinates, is the sequence $(\frac{1}{2}, \dots, \frac{1}{2})$ (this has length $m = \dim E$) and $\rho = (\frac{2m-1}{2}, \frac{2m-3}{2}, \dots, \frac{1}{2})$. Then $\Tor_i(\Delta, \bC)$ is a sum of representations of $\fgl(E)$ whose highest weights $\mu$ are dominant and of the form $w(\lambda + \rho) - \rho$ for $w$ in the Weyl group of type $\rB_m$ (i.e., the group of signed permutations of sequences of length $m$).

First, we claim that this happens if $\mu$ is of the form $(-\nu_m + \frac{1}{2}, \dots, -\nu_1 + \frac{1}{2})$ where $\nu \in Q_1$ and in this case $\ell(w) = |\nu|/2$; this follows from \cite[Proof of Lemma 4.6]{exceptional}. Second, the total number of $\mu$ is $2^m$ (the size of the Weyl group of $\fso(E')$ divided by the Weyl group of $\fgl(E)$), and the total number of $\nu \in Q_1$ with $\ell(\nu) \le m$ is also $2^m$ (Remark~\ref{rmk:Q1-num}). Using \eqref{eqn:wedge-plethysm}, we get the result up to duality; since $\Delta \cong \Delta^*$, we can finish using Poincar\'e duality on Tor.

For the last part, we can use this Tor calculation to build an acyclic complex of free $\rU(\fg(E))$-modules with terms $\Tor_i(\Delta, \bC) \otimes \rU(\fg(E))$. Now twist the complex by $\bC_{(\frac{1}{2}, \dots, \frac{1}{2})}$.
\end{proof}

If we apply transpose duality to $\fg(E)$, we get the free $2$-step nilpotent Lie superalgebra on $E$: $\fg^\dagger(E) = E \oplus \Sym^2(E)$. As a vector space, its universal enveloping algebra is 
\begin{align} \label{eqn:char-spin-univ}
\rU(\fg^\dagger(E)) = \bigwedge^\bullet E \otimes \Sym(\Sym^2(E)) = \bigoplus_{\lambda} \bS_\lambda E.
\end{align}
The above calculations for $\fg(E)$ immediately give calculations for $\fg^\dagger(E)$:

\begin{proposition} \label{prop:spin-koszul}
The minimal free resolution $\ol{\bK}(\fg^\dagger(E))_\bullet$ of $\bC$ over $\rU(\fg^\dagger(E))$ has terms
\[
\ol{\bK}(\fg^\dagger(E))_i = \rU(\fg^\dagger(E)) \otimes \bigoplus_{\substack{\lambda = \lambda^\dagger, \\ 2i = |\lambda| + \rank(\lambda)}} \bS_\lambda E.
\]
\end{proposition}

\begin{proposition} \label{prop:nabla-res}
There is a free resolution $\bF_\bullet$ of $\nabla$ with terms $\bF_i = \rU(\fg^\dagger(E)) \otimes \bigwedge^i(\bigwedge^2(E))$.
\end{proposition}

Recall the map $m \colon \rU(\fg^\dagger(E)) \otimes \fM^- \to \fM$ defined in \S\ref{sec:howe-dual}.

\begin{proposition} \label{prop:spin-sep-vars}
If $2\dim E \le \dim V$, then as representations of $\GL(E)$, both $\rU(\fg^\dagger(E)) \otimes \fM^-$ and $\fM$ have finite multiplicities and are isomorphic. In particular, $m$ is an isomorphism, $\fM$ is a free module over $\rU(\fg^\dagger(E))$, and $\Tor^{\rU(\fg^\dagger(E))}_i(\fM, \bC) = 0$ for $i>0$.
\end{proposition}

\begin{proof}
Since $m$ is surjective by Proposition~\ref{prop:mu-minus}(a), it suffices to show that their characters are the same, and we use $[-]$ to denote characters. From \cite[Proposition 4.8]{exceptional} and Proposition~\ref{prop:mu-minus}(b), we have
\[
[\fM^-] = [\Sym(E \otimes V)] \sum_{i \ge 0} (-1)^i \sum_{\substack{\alpha = \alpha^\dagger\\ 2i = |\alpha| + \rank(\alpha)}} [\bS_\alpha E \otimes \Delta].
\]
Multiply both sides by $[\rU(\fg^\dagger(E))]$. Then the right hand side becomes
\[
\left( \sum_{i \ge 0} (-1)^i [\ol{\bK}(\fg^\dagger(E))_i] \right) [\Sym(E \otimes V) \otimes \Delta] = [\Sym(E \otimes V) \otimes \Delta] = [\fM],
\]
where in the first equality, we used that $\ol{\bK}(\fg^\dagger(E))_\bullet$ is a resolution of $\bC$ (Proposition~\ref{prop:spin-koszul}). In conclusion, $[\rU(\fg^\dagger(E)) \otimes \fM^-] = [\fM]$, so we are done.
\end{proof}

\begin{corollary}
In $\rK(\Rep^\spin(\fso(2\infty + 1)))$, we have 
\[
[\bS_\lambda(\bV) \otimes \Delta] = \sum_{\alpha, \beta} c^\lambda_{\alpha, \beta} [\Delta_\beta].
\]
Similarly, in $\rK(\Rep^\osc(\fosp(1|2\infty)))$, we have 
\[
[\bS_\lambda(\bV) \otimes \nabla] = \sum_{\alpha, \beta} c^\lambda_{\alpha, \beta} [\nabla_\beta].
\]
\end{corollary}

\begin{proof}
It suffices to prove this for $V$ with $\dim V$ finite, but sufficiently large. Pick $E$ with $\dim E \gg 0$ and keep $\dim V \ge 2 \dim E$. From Proposition~\ref{prop:spin-sep-vars}, we get $[\rU(\fg^\dagger(E)) \otimes \fM^-] = [\fM]$. Then $\bS_\lambda(\bV) \otimes \Delta$ is the coefficient of $\bS_\lambda E$ in $\fM = \Sym(E \otimes V) \otimes \Delta$ using the Cauchy identity $\Sym(E \otimes V) = \bigoplus_\lambda \bS_\lambda(E) \otimes \bS_\lambda(V)$ \cite[(3.13)]{expos}. The desired sum is the coefficient of $\bS_\lambda E$ in $\rU(\fg^\dagger(E)) \otimes \fM^-$ (use \eqref{eqn:char-spin-univ} and Proposition~\ref{prop:mu-minus}).

The second formula follows from the first one by applying transpose duality.
\end{proof}

\begin{proposition}
\addtocounter{equation}{-1}
\begin{subequations}
We have an injective resolution $\Delta_\lambda \to \bF(\lambda)_\bullet$ in $\Rep^\spin(\fso(2\infty+1))$ where
\begin{align} \label{eqn:spin-res}
\bF(\lambda)_i = \bigoplus_{\substack{\mu = \mu^\dagger\\ 2i = |\mu|+\rank(\mu)}} \bS_{\lambda / \mu} \bV \otimes \Delta,
\end{align}
and an injective resolution $\nabla_\lambda \to \bF(\lambda)_\bullet$ in $\Rep^\osc(\fosp(1|2\infty))$ where
\begin{align} \label{eqn:osc-res}
\bF(\lambda)_i = \bigoplus_{\substack{\mu = \mu^\dagger\\ 2i = |\mu|+\rank(\mu)}} \bS_{\lambda / \mu} \bV \otimes \nabla.
\end{align}
\end{subequations}
\end{proposition}

\begin{proof}
The first resolution is obtained by taking the dual of the resolution in \cite[Remark 4.9]{exceptional} and taking the limit $\dim V \to \infty$. The second resolution is obtained by applying transpose duality (see \S\ref{ss:transpose}) to the first one.
\end{proof}

\section{Derived specialization} \label{sec:derivedspec}

\subsection{Modification rule}

Applying Theorem~\ref{thm:spin-univ} to the spinor representation of $\Pin(N)$, we obtain a left-exact specialization functor
\[
\Gamma_N \colon \Rep^\spin(\fso(2\infty+1)) \to \Rep(\Pin(N)).
\]
Now let $\Sp(N)$ be the usual symplectic group if $N=2n$ is even, and the orthosymplectic group $\OSp(1 \mid 2n)$ if $N=2n+1$ is odd. Let $\Mp(N)$ be its metaplectic cover. Applying Theorem~\ref{thm:osc-univ} to the oscillator representation, we obtain a left-exact specialization functor
\[
\Gamma_N \colon \Rep^\osc(\fosp(1|2\infty)) \to \Rep(\Mp(N)).
\]
The goal of this section is to calculate the right derived functors of $\Gamma_N$ on the simple objects $\Delta_\lambda$ and $\nabla_\lambda$. Write $\cC^\spin$ and $\cC^\osc$ in place of $\Rep^\spin(\fso(2\infty+1))$ and $\Rep^\osc(\fosp(1|2\infty))$.

\begin{definition} \label{defn:spin-mod}
Let $n = \lfloor N / 2 \rfloor$, so that $N \in \{2n, 2n+1\}$. Given a partition $\alpha$ with more than $n$ parts, let $R_\alpha$ be a border strip of length $2\ell(\alpha) - N - 1$ if it exists, and let $c(R_\alpha)$ be the number of its columns. Then we set $\beta = \alpha \setminus R_\alpha$, and $j_N(\alpha) = j_N(\beta) + c(R_\alpha)$ and $\tau_N(\alpha) = \tau_N(\beta)$. If it does not exist, we define $j_N(\alpha) = \infty$ and leave $\tau_N(\alpha)$ undefined.
\end{definition}

\begin{remark} \label{rmk:weyl-group-modrule}
We can rephrase this border strip rule in terms of a Weyl group action on $\alpha^\dagger$ using $\rho = -\frac{1}{2}(N+1, N+3, N+5, \dots)$ and the Weyl group $W(\rB\rC_\infty)$. More specifically, if $\tau_N(\alpha)$ is defined, then there exists a unique $w \in W(\rB\rC_\infty)$ so that $w(\alpha^\dagger + \rho) - \rho = \tau_N(\alpha)^\dagger$, and we have $\ell(w) = j_N(\alpha)$. If $\tau_N(\alpha)$ is undefined, then there exists a non-identity $w$ so that $w(\alpha^\dagger + \rho) - \rho = \alpha^\dagger$.
This is shown in \cite[Proposition 3.5]{lwood} when $N=2n+1$ is odd, and the same proof applies when $N=2n$ is even (the definitions of $\tau_{2n+1}$ and $j_{2n+1}$ given here coincide with the definitions of $\tau_{2n}$ and $i_{2n}$ given in \cite[\S 3.4]{lwood}).
\end{remark}

\begin{proposition} \label{prop:spin-modrule}
The modification rule above calculates the Euler characteristic of the derived specialization: in $\rK(\Spin(N))$, we have $\sum_{i \ge 0} (-1)^i [\rR^i \Gamma_N \Delta_\lambda] = (-1)^{j_N(\lambda)} [V_{\tau_N(\lambda) + \delta}]$.
\end{proposition}

\addtocounter{equation}{-1}
\begin{subequations}
\begin{proof}
We will use the Weyl group description in Remark~\ref{rmk:weyl-group-modrule}. The Grothendieck group of $\cC^\spin$ is isomorphic to the ring of symmetric functions $\Lambda$ in such a way that the injective module $\bS_\lambda \otimes \Delta$ corresponds to the Schur function $s_\lambda$. For each $\lambda$, define $s^\square_\lambda$ by
\begin{align}
s^\square_\lambda = \sum_{\mu = \mu^\dagger} (-1)^{(|\mu|+\rank(\mu))/2} s_{\lambda / \mu}.
\end{align}
This is the basis (with the same notation) defined in \cite[(4.1)]{shimozono}. By taking the Euler characteristic of \eqref{eqn:spin-res}, we deduce that the simple objects $\Delta_\lambda$ map to the basis $s_\lambda^\square$.

Write $e_r = s_{1^r}$ and $e^\square_r = s^\square_{1^r} = e_r - e_{r-1}$. Write $\be^\square_r$ for the column vector $(e^\square_r, e^\square_{r+1} + e^\square_{r-1}, e^\square_{r+2} + e^\square_{r-2}, \dots)^T$. Applying the involution $\omega_\square$ \cite[(5.17)]{shimozono} to the determinantal formula \cite[Proposition 12]{shimozono}, we get 
\begin{align} \label{eqn:slambda-det}
s_\lambda^\square = \det \begin{bmatrix} \be^\square_{\lambda_1^\dagger} & \be^\square_{\lambda_2^\dagger - 1} & \cdots & \be^\square_{\lambda_{\ell}^{\dagger} - (\ell - 1)} \end{bmatrix}
\end{align}
where $\ell = \ell(\lambda^\dagger) = \lambda_1$, and we have taken the first $\ell$ entries from each row $\be_{\lambda_i^\dagger - (i-1)}^\square$ to get a square matrix. If $2\ell(\lambda) \le N$, this formula is also valid for the character of $V_{\lambda + \delta}$ if we interpret $e_r$ as the character of $\bigwedge^r \bC^N \otimes \Delta$.

Finally, we need to understand what happens when $2\ell(\lambda) > N$. We have $\bigwedge^r \bC^N \cong \bigwedge^{N-r} \bC^N$ as representations of $\Spin(N)$, which implies that $e_r = e_{N-r}$ in $\rK(\Spin(N))$. In particular, $e^\square_r = -e_{N+1-r}^\square$, where we interpret $e^\square_r=0$ for $r<0$, and $\be^\square_r = -\be^\square_{N+1-r}$ in $\rK(\Spin(N))$. In particular, the operations of replacing $\be_{\lambda_i^\dagger - (i-1)}^\square$ with $-\be_{N+1-\lambda_i^\dagger + (i-1)}^\square$ and rearranging rows of \eqref{eqn:slambda-det} corresponds to the action of $W(\rB\rC_\infty)$ on $\lambda^\dagger - \rho$ as described in Remark~\ref{rmk:weyl-group-modrule}. Each of these two types of swaps introduces a negative sign to the determinantal formula, which is exactly the sign function on $W(\rB\rC_\infty)$. This finishes the proof.
\end{proof}
\end{subequations}

\subsection{Spinors}

Maintain the notation from \S \ref{s:littlewood}. Recall that $N = \dim V$.

\begin{theorem} \label{thm:spin-derived}
We have 
\[
\Tor^{\rU(\fg^\dagger(E))}_i(\fM, \bC) = \bigoplus_{j_N(\lambda) = i} \bS_\lambda(E) \otimes V_{\tau_N(\lambda) + \delta}
\]
as representations of $\GL(E) \times \Pin(V)$.
\end{theorem}

We defer the proof to \S\ref{sec:spin-even} and \S\ref{sec:spin-odd}.

\begin{corollary}  \label{cor:spin-kostant}
We have
\[
\rH_i(\fg^\dagger(E); \fM_\lambda) = \bigoplus_{\substack{\mu\\ j_N(\mu) = i\\ \tau_N(\mu) = \lambda}} \bS_\mu(E).
\]
\end{corollary}

\begin{proof}
This is obtained from Theorem~\ref{thm:spin-derived} by taking the $V_{\lambda + \delta}$-isotypic component. 
\end{proof}

\begin{corollary}
The right-derived functors of $\Gamma_N \colon \cC^\spin \to \Rep(\Pin(V))$ on simple objects are
\[
(\rR^i \Gamma_N)(\Delta_\lambda) = \begin{cases} V_{\tau_N(\lambda) + \delta} & \mathrm{if}\ i = j_N(\lambda)\\
0 & \mathrm{else}\end{cases}.
\]
\end{corollary}

\begin{proof}
From \eqref{eqn:spin-res}, we have an injective resolution $0 \to \Delta_\lambda \to \bF(\lambda)_\bullet \to 0$ in $\cC^\spin$ where 
\[
\bF(\lambda)_i = \bigoplus_{\substack{\mu = \mu^\dagger\\ 2i = |\mu| + \rank(\mu)}} \bS_{\lambda/\mu}(\bV) \otimes \Delta.
\]
Recall that $\ol{\bK}(\fg^\dagger(E))_i = \bigoplus_{\substack{\mu = \mu^\dagger\\ 2i=|\mu| + \rank(\mu)}} \bS_\mu(E)$ and that $\fM = \Sym(E \otimes V) \otimes V_\delta$ as a $\GL(E)$-representation. So when we apply $\Gamma_N$, we replace $\bV$ by $V$, and we see that $\bF(\lambda)_\bullet$ is the dual of the $\bS_\lambda E$-isotypic component of $\ol{\bK}(\fg^\dagger(E))_\bullet \otimes_{\rU(\fg^\dagger(E))} \fM$. The homology of the latter is computed by Theorem~\ref{thm:spin-derived}, so we get the desired result.
\end{proof}
 
\subsection{Oscillators}

Maintain the notation from \S\ref{s:littlewood}. Since \eqref{eqn:spin-res} and \eqref{eqn:osc-res} have the exact same form, we conclude that the modification rule described in Definition~\ref{defn:spin-mod} also calculates the Euler characteristic of the derived specialization from $\cC^\osc$:

\begin{proposition} \label{prop:osc-modrule}
We have an identity of characters $\sum_{i \ge 0} (-1)^i [\rR^i \Gamma_N \nabla_\lambda] = (-1)^{j_N(\lambda)} [V_{\tau_N(\lambda) + \eta}]$.
\end{proposition}

\begin{theorem} \label{thm:osc-derived}
We have 
\[
\Tor^{\rU(\fg(E))}_i(\fN, \bC) = \bigoplus_{j_N(\lambda) = i} \bS_\lambda(E) \otimes V_{\tau_N(\lambda) + \eta}
\]
as representations of $\GL(E) \times \Mp(V)$.
\end{theorem}

We defer the proof to \S\ref{sec:osc-even} and \S\ref{sec:osc-odd}.

\begin{corollary}  \label{cor:osc-kostant}
We have
\[
\rH_i(\fg(E); \fN_\lambda) = \bigoplus_{\substack{\mu\\ j_N(\mu) = i\\ \tau_N(\mu) = \lambda}} \bS_\mu(E).
\]
\end{corollary}

\begin{proof}
This is obtained from Theorem~\ref{thm:osc-derived} by taking the $V_{\lambda + \eta}$-isotypic component. 
\end{proof}

\begin{corollary}
The right-derived functors of $\Gamma_N \colon \cC^\osc \to \Rep(\Mp(N))$ on simple objects are given by
\[
(\rR^i \Gamma_N)(\nabla_\lambda) = \begin{cases} V_{\tau_N(\lambda) + \eta} & \mathrm{if}\ i = j_N(\lambda)\\
0 & \mathrm{else}\end{cases}.
\]
\end{corollary}

\begin{remark}
In Corollary~\ref{cor:spin-kostant} and~\ref{cor:osc-kostant}, we see examples of modules where the homology with respect to a parabolic subalgebra is multiplicity-free with respect to the action of the Levi subalgebra. This is consistent with the guess that these modules are Kostant modules in the sense of \cite[\S 5.2]{ehp}, i.e., they have resolutions by direct sums of parabolic Verma modules. We did not attempt to make the combinatorial translation for $\fN_\lambda$ to apply \cite[Theorem 5.15]{ehp}, and we could not find a reference for the Lie superalgebra case, so we have not tried to resolve this guess.
\end{remark}

\subsection{Determinantal ideals} \label{sec:determinantal}
Maintain the notation from \S\ref{s:littlewood}. 

\begin{proposition}
$\fN_\emptyset$ is the quotient of $\rU(\fg(E))$ by the left ideal generated by $\bigwedge^{N+1}(E)$. As a representation of $\GL(E)$, we have
\[
\fN_\emptyset = \bigoplus_{\ell(\lambda) \le N} \bS_\lambda(E).
\]
\end{proposition}

\begin{proof}
The first statement follows from Corollary~\ref{cor:osc-kostant}. 

For the second statement, note that $\rU(\fg(E))$ degenerates to $\Sym(E) \otimes \Sym(\bigwedge^2 E) \cong \bigoplus_\lambda \bS_\lambda(E)$. The ideal structure of this latter algebra is worked out in \cite[Theorem 4.2]{daszkiewicz}. In particular, it is shown that $\bS_\mu(E)$ is in the ideal generated by $\bS_\lambda(E)$ if $\mu \supseteq \lambda$ and the number of odd length columns of $\mu$ is at least the number of odd length columns of $\lambda$. The latter condition is vacuous when $\lambda = (1^{2n})$, so proves our claim when $N$ is odd. When $N$ is even, it is enough to know that $\bigwedge^{N+1}(E)$ generates $\bigwedge^{N+2}(E)$ in $\rU(\fg(E))$. If this were not true, then $\bigwedge^{N+2}(E)$ would appear in $\rH_2(\fg(E); \fN_\emptyset)$, but this is ruled out by direct inspection using Corollary~\ref{cor:osc-kostant}.
\end{proof}

In particular, we may think of $\fN_\emptyset$ as the quotient by a ``determinantal ideal''. The minimal free resolution over $\rU(\fg(E))$ is very close to the classical cases, as discussed in \cite[\S 6]{weyman}: the connection is that a closely related modification rule describes the minimal free resolution in the classical case, see \cite[Remark 3.7]{lwood}. Explicitly, for every representation $\bS_\lambda(E)$ which appears in $\rH_\bullet(\fg(E); \fN_\emptyset)$, $\lambda$ has the following form: there exists $r \ge 0$ and a partition $\alpha$ with $\ell(\alpha) \le r$ such that $\lambda = (r + \alpha_1, \dots, r + \alpha_r, r^N, \alpha^\dagger_1, \alpha^\dagger_2, \dots)$ (here $r^N$ just means $r$ repeated $N$ times). This $\bS_\lambda(E)$ appears exactly once and in homological degree $r(r+1)/2 + |\alpha|$.

We have a parallel story for $\fM_\emptyset$ using instead Corollary~\ref{cor:spin-kostant}:

\begin{proposition}
$\fM_\emptyset$ is the quotient of $\rU(\fg^\dagger(E))$ by the left ideal generated by $\bigwedge^{N+1}(E)$. As a representation of $\GL(E)$, we have
\[
\fM_\emptyset = \bigoplus_{\ell(\lambda) \le N} \bS_\lambda(E).
\]
\end{proposition}

As representations of $\GL(E)$, we have $\rH_i(\fg(E); \fN_\emptyset) \cong \rH_i(\fg^\dagger(E); \fM_\emptyset)$ since their combinatorial descriptions are the same.

\section{Geometric constructions}

In this last section, we use some geometric constructions and prove the remaining unproven statements from the previous section.

\subsection{Preliminaries}

We will need some results from \cite{lwood}; first we recall some notation from \cite[\S 2.2]{lwood}. Let $n$ be a positive integer. Let $\cU$ be the set of integer sequences $(a_1, a_2, \dots)$. Let $\lambda$ be a partition with $\ell(\lambda) \le n$. Given another partition $\mu$, we write $(\lambda \mid \mu)$ for the integer sequence $(\lambda_1, \dots, \lambda_n, \mu_1, \mu_2, \dots) \in \cU$. Let $\fS$ be the finitary infinite symmetric group; it acts on $\cU$ and is generated by adjacent transpositions $s_i$ which swap the $i$th and $(i+1)$st positions. The length $\ell(w)$ of $w$ is the shortest expression $w = s_{i_1} \cdots s_{i_\ell(w)}$. Define $\rho = (0,-1,-2,\dots)$. We define a modified action of $\fS$ on $\cU$ by $w \bullet \alpha = w(\alpha + \rho) - \rho$. We say that $(\lambda \mid \mu)$ is {\bf regular} if its stabilizer subgroup under the modified action is trivial. In that case, there is a unique $w \in \fS$ so that $w\bullet \alpha$ is a partition. Define
\begin{displaymath}
\begin{split}
S_1(\lambda) &= \{ \textrm{$\mu = \mu^\dagger$ such that $(\lambda \mid \mu)$ is regular} \} \\
S_2(\lambda) &= \{ \textrm{partitions $\alpha$ such that $\tau_{2n}(\alpha)=\lambda$} \}.
\end{split}
\end{displaymath}

\begin{lemma} \label{lem:spin-even}
Let $\mu$ be a non-zero partition in $S_1(\lambda)$ and let $\nu$ be the partition obtained by removing the first row and column of $\mu$.  Then $\nu$ also belongs to $S_1(\lambda)$.  Furthermore, let $w$ $($resp.\ $w')$ be the element of $W$ such that $\alpha=w \bullet (\lambda \mid \mu)$ $($resp.\ $\beta=w' \bullet (\lambda \mid \nu))$ is a partition.  Then a border strip $R_\alpha$ of length $2\ell(\alpha) - 2n - 1$ in $\alpha$ exists and we have the following identities:
\begin{displaymath}
\vert R_{\alpha} \vert= 2\ell(\alpha) - 2n - 1 = 2\mu_1-1, \qquad
\alpha \setminus R_{\alpha}=\beta, \qquad
c(R_{\alpha})=\mu_1+\ell(w')-\ell(w).
\end{displaymath}
\end{lemma}

\begin{proof}
This is \cite[Lemma 4.13]{lwood}\footnote{There is a typo in the published version: ``$|R_\alpha|=2\mu_1 + 1$'' should be ``$|R_\alpha| = 2\mu_1 - 1$.''}, but the reader should be aware that in \cite[Lemma 4.13]{lwood}, the function $\tau_{2n+1}$ is used, and the discrepancy comes from the difference in how it is defined in \cite[\S 4.4]{lwood}.
\end{proof}

\begin{proposition} \label{prop:even-spin-bott}
There is a unique bijection $S_1(\lambda) \to S_2(\lambda)$ under which $\mu$ maps to $\alpha$ if there exists $w \in \fS$ such that $w \bullet (\lambda \vert \mu)=\alpha$; in this case, $\ell(w)+j_{2n}(\alpha)=\tfrac{1}{2}(\vert \mu \vert+\rank(\mu))$ and $\tau_{2n}(\alpha) = \lambda$.
\end{proposition}

\begin{proof}
We first define the map $S_1(\lambda) \to S_2(\lambda)$ by induction on $|\mu|$. If $|\mu|=0$, define $\mu \mapsto \lambda$ and there is nothing to prove. Otherwise, let $\nu$ be the partition obtained from $\mu$ by removing the first row and column. By Lemma~\ref{lem:spin-even}, $\nu \in S_1(\lambda)$ and we have $w' \bullet (\lambda|\nu) = \beta \in S_2(\lambda)$ for a unique choice of $w'$, and the equality 
\[
\ell(w') + j_{2n}(\beta) = \frac{|\nu|+\rank(\nu)}{2}
\]
holds. By Lemma~\ref{lem:spin-even}, this implies 
\[
\ell(w) - \mu_1 + c(R_\alpha) + j_{2n}(\beta) = \frac{|\nu|+\rank(\nu)}{2}
\]
where $\beta = \alpha \setminus R_\alpha$. Rearranging this identity we get 
\[
\ell(w) + j_{2n}(\alpha) = \frac{|\mu| + \rank(\mu)}{2}
\]
as desired. To show that this is a bijection, use the argument in \cite[Proposition 3.10]{lwood}.
\end{proof}

Let $E$ be a vector space and let $X$ be the Grassmannian of rank $n$ quotients of $E$. We have a tautological exact sequence
\[
0 \to \cR \to E \otimes \cO_X \to \cQ \to 0
\]
where $\cQ$ is the universal quotient bundle. The Borel--Weil--Bott theorem \cite[\S 4.1]{weyman} is then:

\begin{theorem}[Borel--Weil--Bott]
Let $\lambda$ be a partition with at most $n$ parts, let $\mu$ be any partition and let $\cV$ be the vector bundle $\bS_{\lambda}(\cQ) \otimes \bS_{\mu}(\cR)$ on $X$.
\begin{itemize}
\item Suppose $(\lambda \mid \mu)$ is regular, and write $w \bullet (\lambda \mid \mu)=\alpha$ for a partition $\alpha$.  Then
\begin{displaymath}
\rH^i(X; \cV) = \begin{cases}
\bS_{\alpha}(E) & \textrm{if $i=\ell(w)$} \\
0 & \textrm{otherwise}
\end{cases}.
\end{displaymath}
\item If $(\lambda \mid \mu)$ is not regular, then $\rH^i(X; \cV)=0$ for all $i$.
\end{itemize}
\end{theorem}

\subsection{Proof of Theorem~\ref{thm:spin-derived} in even case $N=2n$} \label{sec:spin-even}

Define a sheaf of Lie superalgebras
\[
\fg^\dagger(E \otimes \cO_X) = \fg^\dagger(E) \otimes \cO_X = (E \oplus \Sym^2(E)) \otimes \cO_X.
\]
It has a Lie subalgebra $\fg^\dagger(\cR) \subset \fg^\dagger(E \otimes \cO_X)$ where $\fg^\dagger(\cR) = \cR \oplus \Sym^2(\cR)$ as a vector bundle. If we take the (minimal) Koszul complex for $\rU(\fg^\dagger(\cR))$ and base change to $\rU(\fg^\dagger(E \otimes \cO_X))$, then by Proposition~\ref{prop:spin-koszul} we get an acyclic complex $\cK_\bullet$ of $\rU(\fg^\dagger(E \otimes \cO_X))$-modules with the terms:
\[
\cK_i = \rU(\fg^\dagger(E \otimes \cO_X)) \otimes \bigoplus_{\substack{\mu = \mu^\dagger\\ 2i = |\mu| + \rank(\mu)}} \bS_\mu \cR.
\]

 Given a partition $\lambda$ with $\ell(\lambda) \le n$, set 
\[
\cF(\lambda)_\bullet = \cK_\bullet \otimes \bS_\lambda \cQ
\]
and $\cM_\lambda = \rH_0(\cF(\lambda)_\bullet)$. We have a left-exact pushforward functor 
\[
p_\ast \colon \rU(\fg^\dagger(E \otimes \cO_X))\text{-Mod} \to \rU(\fg^\dagger(E))\text{-Mod}.
\]
Let $\bF(\lambda)_\bullet = \rR p_* (\cF(\lambda)_\bullet)$. This is a minimal complex of free $\rU(\fg^\dagger(E))$-modules. The Borel--Weil--Bott theorem implies that $\cM_\lambda$ has no higher cohomology, so $\rH_i(\bF(\lambda)_\bullet) = 0$ for $i \ne 0$ and $\rH^0(X; \cM_\lambda) = \rH_0(\bF(\lambda)_\bullet)$ as $\rU(\fg^\dagger(E))$-modules; we denote this common module by $M_\lambda$.

\begin{proposition} \label{prop:spin-tor}
\[
\bF(\lambda)_i = \rU(\fg^\dagger(E)) \otimes \bigoplus_{\substack{\nu\\ \tau_{2n}(\nu) = \lambda\\ j_{2n}(\nu) = i}} \bS_\nu E.
\]
\end{proposition}

\begin{proof}
This follows from Proposition~\ref{prop:even-spin-bott} and the Borel--Weil--Bott theorem.
\end{proof}

\begin{proposition} \label{prop:spin-id}
$M_\lambda \cong \fM_\lambda = \hom_{\Pin(V)}(V_{\lambda + \delta}, \fM)$.
\end{proposition}

\begin{proof}
Combining Propositions~\ref{prop:spin-modrule} and~\ref{prop:spin-tor}, we get that both $M_\lambda$ and $\fM_\lambda$ are representations of $\GL(E)$ such that the multiplicities of each irreducible representation agree and are finite. By Proposition~\ref{prop:mu-minus}, we have a surjection 
\[
\pi_\lambda \colon \bS_\lambda(E) \otimes \rU(\fg^\dagger(E)) \to \fM_\lambda.
\] 
Using Proposition~\ref{prop:spin-tor}, we have a minimal presentation
\[
\bS_{(\lambda, 1^{N+1-2\ell(\lambda)})}(E) \otimes \rU(\fg^\dagger(E)) \to \bS_\lambda(E) \otimes \rU(\fg^\dagger(E)) \to M_\lambda \to 0.
\]
By Pieri's rule, $\bS_{(\lambda, 1^{N+1-2\ell(\lambda)})}(E)$ appears with multiplicity $1$ in $\bS_\lambda(E) \otimes \rU(\fg^\dagger(E))$ and hence does not appear in $M_\lambda$, and so the same is true for $\fM_\lambda$. In particular, $\pi_\lambda$ factors through a surjection $M_\lambda \to \fM_\lambda$. Since they are isomorphic as $\GL(E)$-representations, this surjection is an isomorphism.
\end{proof}

\subsection{Proof of Theorem~\ref{thm:spin-derived} in odd case $N=2n+1$} \label{sec:spin-odd}

Define a sheaf of Lie superalgebras
\[
\fg^\dagger(E \otimes \cO_X) = \fg^\dagger(E) \otimes \cO_X = (E \oplus \Sym^2(E)) \otimes \cO_X.
\]
Similarly, define the sheaf of Lie superalgebras $\fg^\dagger(\cR) = \cR \oplus \Sym^2(\cR)$. Using Proposition~\ref{prop:nabla-res}, we can construct an acyclic locally free complex $\cK_\bullet$ whose terms are:
\begin{align} \label{eqn:spin-odd-Ki}
\cK_i = \rU(\fg^\dagger(E \otimes \cO_X)) \otimes \bigoplus_{\substack{\mu \in Q_{-1}\\ |\mu| = 2i}} \bS_\mu \cR.
\end{align}

Given a partition $\lambda$ with $\ell(\lambda) \le n$, set 
\[
\cF(\lambda)_\bullet = \cK_\bullet \otimes \bS_\lambda \cQ
\]
and $\cM_\lambda = \rH_0(\cF(\lambda)_\bullet)$. We have a left-exact pushforward functor 
\[
p_\ast \colon \rU(\fg^\dagger(E \otimes \cO_X))\text{-Mod} \to \rU(\fg^\dagger(E))\text{-Mod}.
\]
Let $\bF(\lambda)_\bullet = \rR p_* (\cF(\lambda)_\bullet)$. This is a minimal complex of free $\rU(\fg^\dagger(E))$-modules. Again, by the Borel--Weil--Bott theorem, $\rH_i(\bF(\lambda)_\bullet) = 0$ for $i\ne 0$ and $\rH^0(X; \cM_\lambda) = \rH_0(\bF(\lambda)_\bullet)$ as $\rU(\fg^\dagger(E))$-modules; we denote this common module by $M_\lambda$.

\begin{proposition} \label{prop:lwood-spin-odd}
\[
\bF(\lambda)_i = \rU(\fg^\dagger(E)) \otimes \bigoplus_{\substack{\nu\\ \tau_{2n+1}(\nu) = \lambda\\ j_{2n+1}(\nu) = i}} \bS_\nu E.
\]
\end{proposition}

\begin{proof}
This follows from \cite[Lemma 3.12]{lwood}, but let us explain the translation. First, $\tau_{2n+1}$ and $j_{2n+1}$, as defined here, are the same as $\tau_{2n}$ and $i_{2n}$, as defined in \cite[\S 3.4]{lwood}. Second, the vector bundle $\bigoplus_{\substack{\mu \in Q_{-1}\\ |\mu|=2i}} \bS_\mu \cR$ used in the definition of $\cK_i$ in \eqref{eqn:spin-odd-Ki} is the same as $\bigwedge^i(\xi)$ as defined in \cite[Lemma 3.12]{lwood} because of \eqref{eqn:wedge-plethysm}.
\end{proof}

\begin{proposition}
$M_\lambda \cong \fM_\lambda = \hom_{\Pin(V)}(V_{\lambda + \delta}, \fM)$.
\end{proposition}

\begin{proof}
Similar to the proof of Proposition~\ref{prop:spin-id}.
\end{proof}

\subsection{Proof of Theorem~\ref{thm:osc-derived} in even case $N=2n$} \label{sec:osc-even}

Define a sheaf of Lie algebras
\[
\fg(E \otimes \cO_X) = \fg(E) \otimes \cO_X = (E \oplus \lw^2(E)) \otimes \cO_X.
\]
It has a Lie subalgebra $\fg(\cR) \subset \fg(E \otimes \cO_X)$ where $\fg(\cR) = \cR \oplus \lw^2(\cR)$ as a vector bundle. If we take the (minimal) Koszul complex for $\rU(\fg(\cR))$ and base change to $\rU(\fg(E \otimes \cO_X))$, then by Proposition~\ref{prop:osc-koszul} we get an acyclic complex $\cK_\bullet$ of $\rU(\fg(E \otimes \cO_X))$-modules with the terms:
\[
\cK_i = \rU(\fg(E \otimes \cO_X)) \otimes \bigoplus_{\substack{\mu = \mu^\dagger\\ 2i = |\mu| + \rank(\mu)}} \bS_\mu \cR.
\]
 
Given a partition $\lambda$ with $\ell(\lambda) \le n$, set 
\[
\cF(\lambda)_\bullet = \cK_\bullet \otimes \bS_\lambda \cQ
\]
and $\cN_\lambda = \rH_0(\cF(\lambda)_\bullet)$. We have a left-exact pushforward functor 
\[
p_\ast \colon \rU(\fg(E \otimes \cO_X))\text{-Mod} \to \rU(\fg(E))\text{-Mod}.
\]
Let $\bF(\lambda)_\bullet = \rR p_* (\cF(\lambda)_\bullet)$. This is a minimal complex of free $\rU(\fg(E))$-modules. Again, it follows from the Borel--Weil--Bott theorem that $\rH_i(\bF(\lambda)_\bullet) = 0$ for $i \ne 0$ and $\rH^0(X; \cN_\lambda) = \rH_0(\bF(\lambda)_\bullet)$ as $\rU(\fg(E))$-modules; we denote this common module by $N_\lambda$.

\begin{proposition} \label{prop:osc-tor}
\[
\bF(\lambda)_i = \rU(\fg(E)) \otimes \bigoplus_{\substack{\nu\\ \tau_{2n}(\nu) = \lambda\\ j_{2n}(\nu) = i}} \bS_\nu E.
\]
\end{proposition}

\begin{proof}
This follows from Proposition~\ref{prop:even-spin-bott} and the Borel--Weil--Bott theorem.
\end{proof}

\begin{proposition} \label{prop:osc-id}
$N_\lambda \cong \fN_\lambda = \hom_{\Mp(V)}(V_{\lambda + \eta}, \fN)$.
\end{proposition}

\begin{proof}
The proof is similar to the proof of Proposition~\ref{prop:spin-id}.
\end{proof}

\subsection{Proof of Theorem~\ref{thm:osc-derived} in odd case $N=2n+1$} \label{sec:osc-odd}

Define a sheaf of Lie algebras
\[
\fg(E \otimes \cO_X) = \fg(E) \otimes \cO_X = (E \oplus \lw^2(E)) \otimes \cO_X.
\]
Similarly, define the sheaf of Lie algebras $\fg(\cR) = \cR \oplus \lw^2(\cR)$. We have a surjection $\rU(\fg(\cR)) \to \Sym(\cR)$ and $\Tor^{\rU(\fg(\cR))}_i(\Sym(\cR), \cO_X) \cong \lw^i(\lw^2(\cR))$ by Lemma~\ref{lem:torcalc}. This corresponds to a locally free resolution of $\Sym(\cR)$ over $\rU(\fg(\cR))$, and if we base change it to $\rU(\fg(E \otimes \cO_X))$ we get an acyclic locally free complex $\cK_\bullet$ whose terms are:
\begin{align*} 
\cK_i = \rU(\fg(E \otimes \cO_X)) \otimes \bigoplus_{\substack{\mu \in Q_{-1}\\ |\mu| = 2i}} \bS_\mu \cR.
\end{align*}

Given a partition $\lambda$ with $\ell(\lambda) \le n$, set 
\[
\cF(\lambda)_\bullet = \cK_\bullet \otimes \bS_\lambda \cQ
\]
and $\cM_\lambda = \rH_0(\cF(\lambda)_\bullet)$. We have a left-exact pushforward functor 
\[
p_\ast \colon \rU(\fg(E \otimes \cO_X))\text{-Mod} \to \rU(\fg(E))\text{-Mod}.
\]
Let $\bF(\lambda)_\bullet = \rR p_* (\cF(\lambda)_\bullet)$. This is a minimal complex of free $\rU(\fg(E))$-modules. Again, from the Borel--Weil--Bott theorem, we have $\rH_i(\bF(\lambda)_\bullet) = 0$ for $i \ne 0$ and $\rH^0(X; \cN_\lambda) = \rH_0(\bF(\lambda)_\bullet)$ as $\rU(\fg(E))$-modules; we denote this common module by $N_\lambda$.

\begin{proposition}
\[
\bF(\lambda)_i = \rU(\fg(E)) \otimes \bigoplus_{\substack{\nu\\ \tau_{2n+1}(\nu) = \lambda\\ j_{2n+1}(\nu) = i}} \bS_\nu E.
\]
\end{proposition}

\begin{proof}
This is the same as the proof of Proposition~\ref{prop:lwood-spin-odd}. 
\end{proof}

\begin{proposition}
$N_\lambda \cong \fN_\lambda = \hom_{\Mp(V)}(V_{\lambda + \eta}, \fN)$.
\end{proposition}

\begin{proof}
Similar to the proof of Proposition~\ref{prop:spin-id}.
\end{proof}

\end{document}